\documentclass[preprint, 11pt]{elsarticle}

\usepackage[T1]{fontenc}
\usepackage[utf8]{inputenc}
\usepackage[hmargin=0.9in]{geometry}
\usepackage[babel,german=quotes]{csquotes}
\usepackage{subcaption}
\usepackage{graphicx}
\usepackage[colorlinks=true,citecolor=blue,urlcolor=blue]{hyperref}
\usepackage{caption}
\usepackage{amsmath}
\usepackage{amsthm}
\usepackage{amssymb}
\usepackage{bm, amsfonts}
\usepackage{xcolor,soul}
\usepackage{fixmath}
\usepackage{tikz}
\usepackage{bm}
\usepackage{titlesec}
\usepackage{multirow}
\usepackage{indentfirst}
\usepackage{mathtools}
\usepackage{float}
\usepackage[normalem]{ulem}
\usepackage{comment}
\useunder{\uline}{\ul}{}

\newcommand{\diff}{\mathrm{d} }
\newcommand{\R}{\mathbb{R}}
\newcommand{\N}{\mathcal{N}}

\newcommand{\x}{\mathbf{x}}
\newcommand{\s}{\mathbf{s}}
\newcommand{\vel}{\mathbf{v}}
\newcommand{\n}{\mathbf{n}}
\newcommand{\f}{\mathbf{f}}

\newcommand{\con}{\mathbf{c}}

\newcommand{\D}{\mathcal{D}}
\newcommand{\dv}{\nabla\cdot}
\newcommand{\Sp}{\mathbb{S}_{d-1}}
\newcommand{\OO}{\bm{\Omega}}
\newcommand{\pz}{\psi^{(0)}}
\newcommand{\po}{\boldsymbol{\psi}^{(1)}}
\newcommand{\pt}{\boldsymbol{\psi}^{(2)}}
\newcommand{\sprod}[2]{\left\langle #1, #2 \right\rangle}

\usepackage{xcolor}

\colorlet{shadecolor}{green}

\titleformat{\paragraph}
{\normalfont\normalsize\itshape}{\theparagraph}{1em}{}
\titlespacing*{\paragraph}
{0pt}{3.25ex plus 1ex minus .2ex}{1.5ex plus .2ex}

\newdefinition{remark}{Remark}
\newtheorem{theorem}{Theorem}
\newtheorem{lemma}{Lemma}

\makeatletter
\def\ps@pprintTitle{%
\let\@oddhead\@empty
\let\@evenhead\@empty
\def\@oddfoot{
\footnotesize\itshape
\ifx\@journal\@empty Elsevier
\else\@journal\fi
\hfill\today
}%
\let\@evenfoot\@oddfoot}
\makeatother

\begin{document}

\begin{frontmatter}
\title{Realizability-preserving monolithic convex limiting in continuous Galerkin discretizations of the $M_1$ model of radiative transfer}
\author[TUD]{Paul Moujaes\corref{cor1}}
\ead{paul.moujaes@math.tu-dortmund.de}

\author[TUD]{Dmitri Kuzmin}
\ead{kuzmin@math.uni-dortmund.de}

\address[TUD]{Institute of Applied Mathematics (LS III), TU Dortmund University\\ Vogelpothsweg 87,
  D-44227 Dortmund, Germany}

\author[WPE,WTZ,DKTK,TUDOPhy]{Christian Bäumer}

\address[WPE]{West German Proton Therapy Centre Essen
(WPE) gGmbH\\ Am Mühlenbach 1, 45147 Essen, Germany}
\address[WTZ]{West German Cancer Center (WTZ), Hufelandstr. 55, 45147 Essen, University Hospital Essen, Essen, Germany}
\address[DKTK]{German Cancer Consortium (DKTK), Hufelandstr. 55, 45147 Essen, Germany }
\address[TUDOPhy]{\\Department of Physics, TU Dortmund University, Otto-Hahn-Str. 4, 44227 Dortmund, Germany}
\ead{Christian.Baeumer@uk-essen.de}

\cortext[cor1]{Corresponding author}

\journal{}

\begin{abstract}
  We discretize the $M_1$ model of radiative transfer using continuous finite elements and propose a tailor-made monolithic convex limiting (MCL) procedure for  enforcing physical realizability. The $M_1$ system of nonlinear balance laws for the zeroth and first moments of a probability distribution function is derived from the linear Boltzmann equation and equipped with an entropy-based closure for the second moment. To ensure hyperbolicity and physical admissibility, evolving moments must stay in an invariant domain representing a convex set of realizable states. We first construct a low-order method that is provably invariant domain preserving (IDP). Introducing intermediate states that represent spatially averaged exact solutions of homogeneous Riemann problems, we prove that these so-called bar states are realizable in any number of space dimensions. This key auxiliary result enables us to show the IDP property of a fully discrete scheme with a diagonally implicit treatment of reactive terms. To achieve high resolution, we add nonlinear correction terms that are constrained using a two-step MCL algorithm. In the first limiting step, local bounds are imposed on each conserved variable to avoid spurious oscillations and maintain positivity of the scalar-valued zeroth moment (particle density). The second limiting step constrains the magnitude of the vector-valued first moment to be realizable. The flux-corrected finite element scheme is provably IDP. Its ability to prevent nonphysical behavior while attaining high-order accuracy in smooth regions is verified in a series of numerical tests. The developed methodology provides a robust simulation tool for dose calculation in radiotherapy.
\end{abstract}


\end{frontmatter}

\section{Introduction}
\label{sec:intro}

Radiative transfer models based on the linear Boltzmann equation (LBE) are widely used in computational radiotherapy  \cite{bedford2019,gifford2006,vassiliev2010}
and other fields of medical physics. The transported variable of the LBE is a \emph{fluence} that depends on space, time, energy, and direction of travel. Mathematically, LBE has the structure of a Fokker-Planck equation for a nonnegative probability density function. Practical use of deterministic LBE models as a healthcare simulation tool is currently restricted by exorbitant computational cost. An efficient alternative is provided by \emph{moment models}, in which dependence on the angular variable is eliminated and the dimensionality of the problem is reduced \cite{schneider2020,schneider2022}. Since such model reduction leads to a system with more unknowns than equations, a closure relation is required to express higher-order moments in terms of the retained ones. In entropy-based closures, the highest moment is modeled by a solution of an entropy optimization problem~\cite{alldredge2012, brunner2000, brunner2001, coulombel2006, frank2012, hauck2011, levermore1996} or an approximation thereof~\cite{chidyagwai2018, levermore1984, minerbo1978, monreal2013, pichard2017}. The reduced models are
strongly nonlinear and physically meaningful only if reconstructed moments correspond to a nonnegative particle distribution.

The inexpensive $M_1$ model \cite{chu2019,monreal2013} has already proven its worth in the context of dose calculations for radiotherapy \cite{barnard2012,birindelli2019,duclous2010,frank2007,pichard2016}. For the underlying closure to be meaningful, the zeroth moment must remain positive, while the magnitude of the first moment must be bounded above by the zeroth moment~\cite{berthon2007, chu2019, kershaw1976,olbrant2012}.
These constraints define the set of admissible states, which forms a convex cone and is referred to as \textit{realizable set}.
To maintain physical consistency, numerical approximations must remain within this set. In the general context of  hyperbolic problems,
discretizations that guarantee this property are referred to as \textit{invariant domain preserving} (IDP)~\cite{guermond2016,kuzmin2023}.

While recent years witnessed significant advances in the development of IDP methods for hyperbolic flow models, the application of these techniques to the $M_1$ model requires careful extensions and further analysis. 
Adapting property-preserving methods to the $M_1$ system poses
additional challenges due to the forcing terms resulting from particle sources as well as scattering and absorption processes.
In the context of discontinuous Galerkin methods, flux and/or slope limiters
can be applied to the discretized
$M_1$ system~\cite{alldredge2015,chu2019,olbrant2012}, but yield unsatisfactory results in some cases~\cite{chidyagwai2018}.

In this work, we extend the monolithic convex limiting (MCL) framework
introduced in~\cite{kuzmin2020} to a continuous finite element discretization
of the inhomogeneous $M_1$ model. The underlying low-order scheme
preserves invariant domains by design.
Key to its derivation are the so-called \textit{bar states}, which represent spatial averages of exact solutions of the homogeneous Riemann problem~\cite{guermond2016}.
Since realizability of the exact Riemann solutions is only proved in one space dimension~\cite{coulombel2006}, we provide an alternative proof to ensure admissibility of the bar states also in the multidimensional case. 
We treat the reactive term that results from absorption and scattering implicitly, while employing explicit strong stability preserving Runge--Kutta (SSP-RK) methods~\cite{gottlieb2001, shu1989}. 
By lumping the discrete reaction operator, we avoid solving a linear system in each forward Euler stage.
The low-order method serves as the foundation for constructing high-resolution IDP schemes for the $M_1$ system. To ensure numerical stability in the vicinity of shocks and steep fronts, we limit the antidiffusive fluxes that recover the high-order target scheme. The proposed MCL strategy constrains each component of a flux-corrected bar state individually before performing a synchronized IDP fix. The involved limiting steps are similar to those of sequential MCL algorithms for the compressible Euler equations \cite{kuzmin2020,kuzmin2023}.

We begin in Section~\ref{sec:M1} by presenting the $M_1$ model and reviewing some physical properties. In Section~\ref{sec:LO}, we design a low-order discretization that is IDP for all physically admissible particle sources.
Scattering and absorption terms are taken into account in a manner
consistent with the requirement of realizability.
In Section~\ref{sec:MCL}, we introduce our customized MCL scheme for the $M_1$ model.
Finally, we present the results of our numerical experiments in Section~\ref{sec:examples} and draw conclusions in Section~\ref{sec:concl}.

\section{$M_1$ moment model} \label{sec:M1}

Let $\psi=\psi(\x,t,\OO)$ denote a probability density (fluence) that
depends on space location $\x\in\D\subset\R^d$, $d\in\{ 1,2,3\}$, time
instant $t\geq 0$, and orientation $\OO\in\Sp$, where
$\Sp = \{\OO\in\R^d: |\OO|= 1\}$ is the unit sphere.
In the context of radiation transport modeling,
$\psi(\x,t,\cdot):\Sp\to\R_+$ represents
 the angular distribution of particles at a fixed space-time
 location $(\x,t)$. In what follows, we write ``$\gneq$'' if we
 assume that $\psi(\x,t,\cdot)\in L^2(\Sp)$ is nonnegative
 with $\|\psi(\x,t,\cdot)\|_{L^1(\Sp)}>0$.

A detailed description of radiative transfer is provided by
LBE models of the form ~\cite{alldredge2012,monreal2013}
\begin{equation}\label{eq:Boltzmann}
    \frac{\partial \psi}{\partial t} + \OO \cdot \nabla\psi = - (\sigma_s + \sigma_a)\psi + \frac{\sigma_s}{4\pi}\int_{\Sp}\psi(\OO')\,\diff\OO'+Q,
\end{equation}
where $Q=Q(\x,t,\OO)$ is a nonnegative source of particles. The absorption and scattering properties of the background material are characterized by 
$\sigma_a\geq 0$ and $\sigma_s\geq0$, respectively.

In principle, approximate solutions to \eqref{eq:Boltzmann} can be
obtained using numerical methods for transport-reaction equations
(see, e.g.,~\cite{godoy2010,hansel2018,yuan2016}). However, the cost of evolving
$\psi(\x,t,\OO)$ is very high considering that the domain
$\D\times \R_+\times\Sp$ is six-dimensional for $d=3$. Therefore,
it is common practice to approximate \eqref{eq:Boltzmann} by nonlinear
evolution equations for $N+1$ angular moments
$$
\psi^{(n)} = \psi^{(n)}(\x,t) = \int_{\Sp}\underbrace{\OO \otimes \cdots \otimes \OO}_{n\text{ times}}
\psi(\x,t,\OO)\,\diff\OO,\qquad n=0,\ldots,N.
$$
The system of equations for $\psi^{(0)},\ldots,\boldsymbol{\psi}^{(N)}$
is referred to as the \emph{$M_N$ model}. In this work, we
focus on the
numerical treatment of the $M_1$ model, i.e., of balance laws that
govern the evolution of
\begin{align}
    \pz = \pz(\x,t) &= \int_{\Sp} \psi(\x,t,\OO)\,\diff\OO\,\in\R,\label{eq:0moment}\\
    \po = \po(\x,t) &= \int_{\Sp} \OO \psi(\x,t,\OO)\,\diff\OO\,\in\R^d.\label{eq:1moment}
\end{align}
The zeroth moment~\eqref{eq:0moment} corresponds to the total particle density, while the first moment~\eqref{eq:1moment} is the momentum density of particle motion. The flux of momentum is given by the second moment
\begin{equation*}
  \pt=\pt(\x,t) = \int_{\Sp} \OO\otimes\OO \psi(\x,t,\OO)\,\diff\OO\,\in\R^{d\times d},
\end{equation*}
which represents a derived quantity and is calculated using a closure
approximation (see below).

The $M_1$ model of radiative transfer is a
nonlinear hyperbolic system of the form
\begin{equation}\label{eq:M1}
    \frac{\partial u}{\partial t} + \dv \f(u) = -\sigma u + q.
\end{equation}
The vector $u$ of conserved quantities and the matrix
 $\f(u)$ of corresponding fluxes are given by
\begin{equation*}
  u=\left(\begin{array}{cc}
         \pz\\
         \po
  \end{array}\right)\in\R^{d+1},
  \qquad
    \f(u) = \left(\begin{array}{cc}
         \po\\
         \pt
    \end{array}\right)\in\R^{d\times (d+1)}.
\end{equation*}
Note that $\pz$ is transported by $\po$, while
$\po$ is transported by $\pt$.
The  diagonal  tensor
$$\sigma = \mathrm{diag}(\sigma_a, \sigma_t,..., \sigma_t)\in \R^{(d+1)\times(d+1)},
\qquad \sigma_t = \sigma_a +\sigma_s$$
 and the source term
$q = (q^{(0)}, \mathbf{q}^{(1)})^\top\in\R^{d+1}$
 are inferred from the linear Boltzmann equation \eqref{eq:Boltzmann}.

For the second moment, we use the standard closure approximation~\cite{levermore1984}
\begin{equation}\label{eq:M1closure}
    \pt = D\left(\vel\right)\pz, \quad \vel = \frac{\po}{\pz},
\end{equation}
where
\begin{equation}\label{eq:D}
    D(\vel) = \frac{1-\chi(|\vel|)}{2} I_d + \frac{3\chi(|\vel|)-1}{2} \frac{\vel \otimes \vel}{|\vel|^2}
\end{equation}
is the Eddington tensor and
\begin{equation}\label{eq:chi}
    \chi(f) = \frac{3+4f^2}{5+2\sqrt{4-3f^2}}
\end{equation}
is the Eddington factor.
Note that for $d = 1$ the Eddington tensor~\eqref{eq:D} reduces to~\eqref{eq:chi}.

\begin{remark}
  In general, the $n$-th moment is transported by the $(n+1)$-st moment.
Thus, the $M_N$ system requires a closure for $\bm{\psi}^{(N+1)} \coloneqq \bm{\psi}^{(N+1)} ( \pz, ...,\bm{\psi}^{(N)})$.
To ensure physical admissibility, the choice of closure approximations
must guarantee that if $\pz, ...,\bm{\psi}^{(N)}$ are 
moments of a nonnegative distribution $\psi$, then so is
$\bm{\psi}^{(N+1)}$. To that end, $\bm{\psi}^{(N+1)}$ can be defined as
the solution of an entropy minimization problem or an approximation thereof~\cite{alldredge2012, chidyagwai2018, coulombel2006, levermore1984, pichard2017}. 
However, solving optimization problems of this kind is costly. 
Moreover, numerical solvers can introduce errors, which may result in a loss of physical admissibility. 
For details, we refer the interested reader to~\cite[Sec. 3.4]{pichard2017}.
\end{remark}

The moments $\pz$ and $\po$ correspond to a nonnegative probability density
$\psi$ if and only if~\cite{kershaw1976}
\begin{equation}\label{eq:m01realizable}
    \pz >0\quad\text{and}\quad f = |\vel|=\frac{|\po|}{\pz}<1.
\end{equation}
If this requirement is met, we
refer to $\pz$ and $\po$ as \textit{realizable} or say that $\psi$ \textit{realizes} $\pz$ and $\po$.
\medskip

In addition to the validity of conditions \eqref{eq:m01realizable}, we assume that  
$\pt$ is defined by \eqref{eq:M1closure}--\eqref{eq:chi} with
\begin{equation*}
    f^2\leq\chi(f)\leq 1\quad \text{for }f\in[0,1).
\end{equation*}
Under these assumptions, Levermore~\cite{levermore1984} has shown that $\pz,\po,$ and $\pt$ correspond to the zeroth, first, and second moments of a nonnegative function, respectively.

In view of~\eqref{eq:m01realizable}, we define the set of physically admissible states for the $M_1$ model~\eqref{eq:M1} as
\begin{equation}\label{eq:realizableset}
\begin{split}
    \mathcal{R}_1 &=\left\{(\pz,\po)^{\top}\in\R^{d+1}: \pz >0, |\po|< \pz\right\}\\
    &= \left\{\int_{\Sp}
\left(\begin{array}{cc}
         1\\
         \OO
  \end{array}\right)
\psi(\OO)\,\diff\OO,\ \ \psi(\OO)\gneq0\right\}.
\end{split}
\end{equation}
This set is a convex cone that is referred to as \textit{realizable set}.
Furthermore, the $M_1$ model is hyperbolic for all $u\in\mathcal{R}_1$~\cite{berthon2007,levermore1996}.
That is, the directional Jacobian of the flux function
\begin{equation}\label{eq:dirJac}
    \f'_{\n}(u) = \frac{\partial}{\partial u}\left( \f(u)\cdot \n\right)\,\in\R^{(d+1)\times(d+1)}
\end{equation}
is diagonalizable with real eigenvalues for all $u\in\mathcal{R}_1$  and all directions $\n\in\Sp$.
However, hyperbo\-licity of the $M_1$ model breaks down on the boundary of the realizable set~\eqref{eq:realizableset} because the directional Jacobian~\eqref{eq:dirJac} is not diagonalizable for $|\po| =\pz$~\cite{chidyagwai2018}.
Therefore, it is essential for the design of numerical schemes to produce solutions that remain in the interior of $\mathcal{R}_1$.

\begin{remark}
  The requirement that $|\po|$ be bounded by $\pz$ is often referred to as
  \textit{flux limiting} condition. This terminology was introduced in the
  frequently cited paper~\cite{kershaw1976}. To avoid confusion with 
  limiting for numerical fluxes, we call $|\po|<\pz$ the
  \textit{realizable velocity} condition.
\end{remark}

\begin{remark}
    Note that the only nonnegative particle distributions $\psi\geq0$ that map to the boundary of the realizable set~\eqref{eq:realizableset} are the trivial distribution $\psi(\OO)\equiv 0$ a.e. on $\Sp$ and Dirac delta distributions on the unit sphere~\cite{kershaw1976}. 
    Clearly, delta distributions do not belong to $L^2(\Sp)$.
\end{remark}

\section{Low-order discretization}
\label{sec:LO}
In algebraic flux correction schemes for hyperbolic problems, invariant
domain preserving low-order methods serve as building blocks for high-order
extensions constrained by limiters \cite{kuzmin2023}. An invariant domain of the $M_1$ model
\eqref{eq:M1} is the realizable set $\mathcal{R}_1$ defined
by~\eqref{eq:realizableset}. In this
section, we design a low-order continuous finite element method
that produces numerical solutions belonging to $\mathcal{R}_1$.

Multiplying the residual of the $M_1$ system~\eqref{eq:M1} by a test function $w$, assuming sufficient regularity, and integrating over the spatial domain $\D\subset\R^d$, we obtain the weak formulation
\begin{equation}\label{eq:weakform}
    \int_\D w\left(\frac{\partial u}{\partial t}+\dv \f(u) + \sigma u - q\right)\,\diff\x = \int_\Gamma w\left(\f(u)\cdot \n - \mathcal{F}(u, \hat u; \n) \right)\,\diff\s,
\end{equation}
where $\mathcal{F}(u,\hat u;\n)$ is a numerical approximation to the normal flux $\f(u)\cdot\n$ across $\Gamma = \partial\D$.
Problem-dependent boundary conditions are imposed in a weak sense 
by choosing the external state $\hat u$ of the  approximate Riemann solver
accordingly.
In this work,  we use the global Lax--Friedrichs flux
\begin{equation*}
    \mathcal{F}(u_L, u_R; \n) = \frac{\f(u_L) + \f(u_R)}{2} \cdot \n -\frac{\lambda_{\max}}{2} (u_R - u_L).
\end{equation*}
The constant $\lambda_{\max}$ is an upper bound for the global maximum wave speed. The wave speeds of the realizable $M_1$ model are bounded above by unity~\cite{berthon2007, chidyagwai2018, olbrant2012}.
Thus, we set $\lambda_{\max} = 1$.

We discretize~\eqref{eq:weakform} in space using the continuous Galerkin (CG) method on a conforming triangulation $\mathcal{T}_h =\{K_1,\ldots,K_{E_h}\}$ consisting of $E_h$ nonoverlapping elements such that $\overline{\D} = \cup_{e=1}^{E_h} K_e$. The vertices of  $\mathcal{T}_h$ are denoted by $\x_1,\ldots,\x_{N_h}$.
The Lagrange basis functions $\varphi_1,\ldots, \varphi_{N_h}$ of
 a piecewise-$\mathbb{P}_1/\mathbb{Q}_1$ finite element approximation have the property that $\varphi_i(\x_j) = \delta_{ij}$. We seek 
\begin{equation}\label{eq:CGapprox}
    u_h(\x,t) = \sum_{j = 1}^{N_h} u_j(t) \varphi_j(\x)
\end{equation}
in the space $V_h=\mbox{span}\{\varphi_1,\ldots, \varphi_{N_h}\}\subseteq H^1(\mathcal D)\cap C(\bar{\mathcal D})$. The flux $\f(u_h)$ is approximated by
\begin{equation}\label{eq:CGapproxf}
    \f_h(u_h) = \sum_{j = 1}^{N_h}\f_j \varphi_j,\quad \f_j = \f(u_j).
\end{equation}

Let $\N_i =\{j\in \{1,\ldots, N_h\}: \mathrm{supp}(\varphi_i) \cap\mathrm{supp}(\varphi_j)\not= \emptyset\}$ and $\N_i^* = \N_i\setminus\{i\}$ denote the computational stencils associated with node $i\in \{1,\ldots, N_h\}$. Denote the 
 $L^2$ scalar products by
\begin{equation*}
    \sprod{u}{v}_\D= \int_\D u v\,\diff\x, \quad \sprod{u}{v}_\Gamma =  \int_\Gamma u v\,\diff\s.
\end{equation*}
Substituting the approximations~\eqref{eq:CGapprox} and~\eqref{eq:CGapproxf} into~\eqref{eq:weakform} with $w = \varphi_i$, we obtain
\begin{equation}\label{eq:stdCG}
    \sum_{j\in\N_i} m_{ij} \frac{\diff u_j}{\diff t} = b_i(u_h,\hat u) - \sum_{j\in\N_i} \left[\f_j\cdot \con_{ij} + m_{ij}^{\sigma} u_j  \right] + s_i.
\end{equation}
The coefficients of this semi-discrete problem are given by
\begin{equation}\label{eq:consistentmass}
    m_{ij} = \sprod{\varphi_i}{\varphi_j}_\D,\quad m_{ij}^\sigma = \sprod{\varphi_i}{\sigma \varphi_j}_\D,\quad \con_{ij} =\sprod{\varphi_i}{\nabla\varphi_j}_\D, 
\end{equation}

\begin{equation}\label{eq:consistentbdr}
    b_i(u_h, \hat u) = \sprod{\varphi_i}{\f(u_h) \cdot \n - \mathcal{F}(u_h, \hat u;\n)}_\Gamma,\quad s_i = \sprod{\varphi_i}{q}_\D.
\end{equation}

To derive a low-order IDP scheme, we
proceed as in~\cite{guermond2016,kuzmin2020,kuzmin2023}.
Using the \textit{partition of unity} property $\sum_{j=1}^{N_h} \varphi_j\equiv 1$ of the Lagrange basis, we approximate $m_{ij}$ and $m_{ij}^\sigma$ by
 $\delta_{ij}m_i$ and $\delta_{ij}m_i^\sigma$ with
\begin{equation*}
    m_i = \sum_{j\in\N_i} m_{ij} = \sprod{\varphi_i}{1}_\D >0,\quad 
    m_i^\sigma = \sum_{j\in\N_i} m_{ij}^\sigma = \sprod{\varphi_i}{\sigma}_D\geq 0.
\end{equation*}
That is, we lump the mass matrices.
Similarly, the boundary term~ $b_i(u_h, \hat u)$ is approximated by
\begin{equation}\label{eq:lumpedbdr}
    \tilde{b}_i(u_i, \hat u) = \sprod{\varphi_i}{\f_i \cdot \n - \mathcal{F}(u_i, \hat u_i;\n)}_\Gamma. 
\end{equation}
To stabilize the CG discretization of $\nabla\cdot\f(u)$, we define the \emph{graph viscosity} coefficients
\begin{equation*}
    d_{ij} = \begin{cases}
        \lambda_{\max}\max\{|\con_{ij}|, |\con_{ji}|\} &\text{if }j\in\N_i^*,\\
        -\sum_{k\in\N_i^*}d_{ik} &\text{if } j=i,\\
        0&\text{otherwise}
    \end{cases}
\end{equation*}
using the maximum speed $\lambda_{\max} = 1$ of the realizable
$M_1$ model. The addition of diffusive fluxes $d_{ij} (u_j -u_i)$
to the lumped counterpart of \eqref{eq:stdCG}
yields the semi-discrete low-order scheme
\begin{equation}\label{eq:sdLO}
    m_i\frac{\diff u_i}{\diff t} = \tilde{b}_i(u_i, \hat u_i) + \sum_{j \in\N_i^*} \left[d_{ij} (u_j -u_i) - (\f_j-\f_i)\cdot \con_{ij}\right] - m_i^{\sigma}u_i + s_i,
\end{equation}
which represents an extension of the
 Lax--Friedrichs method to continuous finite elements \cite{kuzmin2023}.
 \smallskip
 
We show the IDP property for a fully discrete version of~\eqref{eq:sdLO} by splitting the remainder of this section into three parts.
First, we analyze the homogeneous system, i.e., \eqref{eq:sdLO}
with $m_{i}^\sigma=0$ and $s_i=0$. Next, we include
$s_i\ne 0$ corresponding to a physically admissible source $q$
in \eqref{eq:M1}. Finally, we show that implicit treatment of
the reactive term $m_i^\sigma u_i$ guarantees the IDP property of 
the fully discrete low-order scheme if the remaining terms
are treated explicitly and the time step is sufficiently small.

\subsection{Homogeneous $M_1$ model}
If $\sigma\equiv 0$ and $q\equiv 0$ in the $M_1$ system~\eqref{eq:M1},
then $m_{i}^\sigma=0$ and $s_i\equiv 0$ in \eqref{eq:sdLO}. Suppose that $i$ is an
internal node. Then
$\tilde{b}_i(u_i, \hat u_i)=0$ and the semi-discrete equation
\eqref{eq:sdLO} reduces to (cf. \cite{guermond2016,kuzmin2020,kuzmin2023})
\begin{equation}\label{eq:LOhomogene}
 \begin{split} 
     m_i\frac{\mathrm{d} u_i}{\mathrm{d} t} &= \sum_{j\in\N_i^*}[ d_{ij} (u_j -u_i) - (\f_j -\f_i)\cdot \con_{ij}]\\ 
     &= \sum_{j\in\N_i^*}[ 2d_{ij} (\overline{u}_{ij} -u_i)],
 \end{split}
\end{equation}
where
\begin{equation}\label{eq:LObarstates}
    \bar u_{ij} = \frac{u_i + u_j}{2} - \frac{(\f_j-\f_i)\cdot \con_{ij}}{2d_{ij}}.
\end{equation}
Let $\n_{ij} = \frac{\con_{ij}}{|\con_{ij}|}$.
As explained, e.g., in \cite{guermond2016}, the low-order \emph{bar state}
$\bar u_{ij}$ can be interpreted as a
space-averaged exact solution $u(\xi,\tau)$ of the projected one-dimensional Riemann problem
\begin{equation}\label{eq:1driemann}
    \frac{\partial u}{\partial t}+\dv(\f(u)\cdot \n_{ij}) = 0,\quad u_0(\xi) =\begin{cases}
        u_i, & \xi <0,\\
        u_j, & \xi >0
    \end{cases}
\end{equation}
at the artificial time $\tau_{ij} =\frac{|\con_{ij}|}{2d_{ij}}$.
Thus, the bar states are realizable if exact solutions to the Riemann problem~\eqref{eq:1driemann} stay in $\mathcal{R}_1$. A proof of the fact that Riemann solutions of the homogeneous $M_1$ model are realizable in one space dimension can be found in~\cite{coulombel2006}. In contrast to the Euler equations, for which the extension of one-dimensional analysis is straightforward~\cite{toro2013}, the multidimensional $M_1$ system requires further investigation because it is not obvious that $\mathcal{R}_1$ is an invariant set for $d\in\{2,3\}$. 
\smallskip

To show that the bar states \eqref{eq:LObarstates} are realizable, we need the following lemma~\cite[Lem. 4.1]{chidyagwai2018}.
\begin{lemma}\label{lem:upm}
    Let $u = (\pz,\po)^\top\in\mathcal{R}_1$, $\pt$ be given by~\eqref{eq:M1closure}, and $\bm\nu\in\R^d$ be an arbitrary vector with $|\bm\nu|\leq 1$. 
    Then, the combination of moments
    $u_\pm = (\pz\pm \po\cdot \bm\nu, \po\pm \pt\cdot\bm\nu)^\top$ is realizable
    for the $M_1$ model. 
\end{lemma}
\begin{proof}
    Assume that $\pz, \po$, and $\pt$ are moments of a nonnegative function $\psi$.
    Then 
    \begin{equation*}
        u_\pm = \left(\begin{array}{cc}
             \pz\pm \po\cdot \bm\nu \\
             \po\pm \pt\cdot\bm\nu
        \end{array}\right)
    \end{equation*}
    are the zeroth and first moments of 
    \begin{equation*}
        \psi_\pm(\OO) = (1\pm \bm\nu\cdot \OO)\psi(\OO),\quad \OO\in\Sp.
    \end{equation*}
    Since $|\bm\nu\cdot \OO|\leq |\bm\nu||\OO|\leq 1$, the so-defined
    $\psi_\pm(\OO)$ is nonnegative.
\end{proof}
Let us now show the realizability of the bar states.
\begin{theorem}\label{thm:IDPbarstates}
    The low-order bar states~\eqref{eq:LObarstates} are realizable if $u_i, u_j\in\mathcal{R}_1$ and $d_{ij}\geq |\con_{ij}|$.
\end{theorem}
\begin{proof}
  Introducing the auxiliary states
   \begin{equation}\label{eq:uibar}
        \overline{u}_i = u_i + \f_i\cdot \frac{\con_{ij}}{d_{ij}}, \quad 
        \overline{u}_j = u_j - \f_j\cdot \frac{\con_{ij}}{d_{ij}},
   \end{equation}
   we notice that
    \begin{equation}\label{eq:barstatesplit}
        \overline u_{ij} = \frac{1}{2}\overline{u}_i + \frac{1}{2}\overline{u}_j.
    \end{equation}
    Recall that the flux of the $n$-th moment is the $(n+1)$-th moment. 
    Therefore, the realizability of $\overline{u}_i$ and $\overline{u}_j$ follows directly from Lemma~\ref{lem:upm} under the assumption that $|\bm\nu| = \frac{|\con_{ij}|}{d_{ij}}\leq 1$ and $u_i,u_j\in\mathcal{R}_1$.
    Owing to the convexity of $\mathcal{R}_1$, we obtain the desired result.
\end{proof}
\begin{remark}
    If the nodal states $u_i$ and $u_j$ are realized by $\psi_i\gneq0$ and $\psi_j\gneq0$, respectively, then the fact that the low-order bar states~\eqref{eq:LObarstates} are moments of
    \begin{equation*}
            \overline{\psi}_{ij} = \frac{\psi_i +\psi_j}{2} -\frac{(\psi_j -\psi_i)\,\OO\cdot\con_{ij}}{2d_{ij}} \gneq 0
        \end{equation*}
can be easily verified using the splitting~\eqref{eq:barstatesplit} and Lemma~\ref{lem:upm}.
\end{remark}

\begin{remark}
Berthon et al.~\cite{berthon2007} split the intermediate state of the HLL Riemann solver in a similar way and determine the diffusion coefficient by direct calculation to enforce realizability.
\end{remark}

If the homogeneous semi-discrete problem \eqref{eq:LOhomogene} is discretized in time
using an explicit SSP-RK method, then it is easy to show that,
 under
a suitable time step restriction,
each forward
Euler stage is IDP w.r.t. the set of realizable moments $\mathcal{R}_1$. Indeed, the updated nodal state
\begin{equation}\label{eq:homoupdate}
\begin{split}
    u_i^{\mathrm{SSP,H}} &= u_i +\frac{\Delta t}{m_i}  \sum_{j \in\N_i^*} 2d_{ij} (\bar u_{ij} -u_i) \\
    &= \left(1- \frac{2\Delta t}{m_i}\sum_{j\in\N^*_i}d_{ij}\right)u_i + \frac{2\Delta t}{m_i}\sum_{j\in\N^*_i}d_{ij}\bar u_{ij}
\end{split}
\end{equation}
is a convex combination of states belonging to $\mathcal{R}_1$, provided
that  $u_j\in\mathcal{R}_1\ \forall j\in\N_i$ and the time step $\Delta t$
satisfies the CFL-like condition (cf. \cite{guermond2016})
\begin{equation}\label{eq:CFL}
	\frac{2\Delta t}{m_i}\sum_{j\in\N^*_i}d_{ij} \leq 1.
\end{equation}
Since $\mathcal{R}_1$ is convex, the explicit update \eqref{eq:homoupdate}
yields a realizable state $u_i^{\mathrm{SSP},H}\in \mathcal{R}_1$.

\begin{remark}
  The lumped boundary term~\eqref{eq:lumpedbdr} can be written in a bar state form similar to~\eqref{eq:LOhomogene}. The realizability of nodal states $u_i$ associated with
  boundary points $\x_i\in\Gamma$
  can then be shown using the same convexity argument. For details, we refer the interested reader to~\cite{hajduk2022diss, kuzmin2023, moujaes2025}.
\end{remark}

\subsection{Particle source discretization}

The zeroth and first moments of a particle source $Q\geq 0$ in the Boltzmann transport equation~\eqref{eq:Boltzmann} constitute the source term $q= (q^{(0)},\mathbf{q}^{(1)})^\top$ of the $M_1$ system~\eqref{eq:M1}.
If we assume that $q\in\overline{\mathcal{R}}_1$, i.e.,
\begin{equation*}
    q^{(0)}\geq 0\quad  \text{and}\quad \left| \mathbf{q}^{(1)}\right|\leq q^{(0)},
\end{equation*}
then the components of $s_i = (s_i^{(0)},\s_i^{(1)})^\top$ satisfy
\begin{equation*}
\begin{split}
    s_i^{(0)} &=  \int_\D \varphi_iq^{(0)} \,\diff\x \geq 0,\\
    \left|\s_i^{(1)}\right| &\leq \int_\D \varphi_i\left|\mathbf{q}^{(1)}\right|\,\diff\x \leq s_i^{(0)}.
\end{split}
\end{equation*}
Adding the contribution of $s_i$ to the forward Euler stage
\eqref{eq:homoupdate}, we obtain
\begin{equation}\label{eq:LOparticlesource}
\begin{split}
    u_i^{\mathrm{SSP,S}} &= u_i + \frac{\Delta t}{m_i}  \left(\sum_{j \in\N_i^*} 2d_{ij} (\bar u_{ij} -u_i) + s_i\right) \\
    &= u_i^{\mathrm{SSP,H}} +\frac{\Delta t}{m_i} s_i,
\end{split}
\end{equation}
where $u_i^{\mathrm{SSP,H}}$ is the solution of the homogeneous problem~\eqref{eq:homoupdate}, which we have shown to be in $\mathcal{R}_1$ if $u_j\in\mathcal{R}_1\ \forall j\in\N_i$.
Obviously, if $\psi_i^{(0),\mathrm{SSP,H}}>0$ and $\left|\bm\psi_i^{(1),\mathrm{SSP,H}}\right| < \psi_i^{(0),\mathrm{SSP,H}}$, then
\begin{equation*}
   \psi_i^{(0),\mathrm{SSP,S}} = \psi_i^{(0),\mathrm{SSP,H}} +\frac{\Delta t}{m_i} s_i^{(0)} >0
\end{equation*}
and 
\begin{align*}
     \left|\bm\psi_i^{(1),\mathrm{SSP,S}} \right| &\leq  \left|\bm\psi_i^{(1),\mathrm{SSP,S}}\right| + \left|\frac{\Delta t}{m_i} \s_i^{(1)}\right|\\
     &<   \psi_i^{(0),\mathrm{SSP,H}} + \frac{\Delta t}{m_i} s_i^{(0)}  \\
     & = \psi_i^{(0),\mathrm{SSP,S}}.
\end{align*}
Therefore, $u_i^{\mathrm{SSP,S}}\in \mathcal{R}_1$ under the CFL-like condition~\eqref{eq:CFL}.

\begin{remark}
    The above analysis shows that the result of the forward Euler stage~\eqref{eq:LOparticlesource} is guaranteed to be admissible even for nontrivial sources $q\in\partial\mathcal{R}_1$ that correspond to Dirac delta distributions and represent perfectly collimated particle beams.
\end{remark}

\subsection{Reactive terms}
Let us now turn our attention to the full inhomogeneous system with $\sigma_s\geq 0$, $\sigma_a \geq 0$ and $q\in\overline{\mathcal{R}_1}$. We discretize~\eqref{eq:sdLO} in time using an SSP-RK scheme in which the reactive term $m_i^{\sigma}u_i$ is treated implicitly, while other terms are treated explicitly. That is, each intermediate stage is of the form
\begin{equation*}
    \left(m_i +\Delta t m_i^{\sigma}\right) u_i^{\mathrm{SSP,R}} = m_i u_i + \Delta t\left(\sum_{j\in\N_i^*}2d_{ij} (\overline{u}_{ij} -u_i) + s_i\right).
\end{equation*}
Since $m_i^\sigma = \mathrm{diag}(m_i^{\sigma_a}, m_i^{\sigma_t},\ldots, m_i^{\sigma_t})$ is a diagonal matrix with nonnegative entries, we have
\begin{equation}\label{eq:fullLO}
\begin{split}
    u_i^{\mathrm{SSP,R}} &= \frac{m_i}{m_i + \Delta t m_i^{\tilde\sigma}}\left[ u_i + \frac{\Delta t}{m_i}\left(\sum_{j\in\N_i^*}2d_{ij} (\overline{u}_{ij} -u_i) + s_i\right)\right]\\
    &= \frac{m_i}{m_i + \Delta t m_i^{\tilde\sigma}} u_i^{\mathrm{SSP}, \mathrm{S}},
\end{split}
\end{equation}
where the value of $\tilde\sigma\in\{\sigma_a,\sigma_t\}$ depends on the component and $u_i^{\mathrm{SSP}, S}\in\mathcal{R}_1$ is given by~\eqref{eq:LOparticlesource}. 
We note that $m_i^{\sigma_t} \geq m_i^{\sigma_a}$ since $\sigma_t = \sigma_a +\sigma_s$.

\begin{lemma}\label{lem:scaledstate}
    Let $u_i = (\pz_i,\po_i)^\top \in\mathcal{R}_1$ be a physically admissible state. 
    Then the scaled state 
    \begin{equation}\label{eq:scaledstate}
        \tilde u_i = \frac{m_i}{m_i + \Delta t m_i^{\tilde\sigma}} u_i = \left(\begin{array}{c}
             \frac{m_i}{m_i + \Delta t m_i^{\sigma_a}} \pz_i\\
             \frac{m_i}{m_i + \Delta t m_i^{\sigma_t}} \po_i
        \end{array}\right)\, \in\mathcal R_1
    \end{equation}
    is admissible for any $\Delta t >0$.
\end{lemma}
\begin{proof}
    Assume that $u_i = (\pz_i,\po_i)^\top \in\mathcal{R}_1$, i.e., $\pz >0$ and $|\po_i|< \pz$.
    To show the admissibility of the state $\tilde u_i = (\tilde\psi_i^{(0)}, \tilde{\bm\psi}_i^{(1)})^\top$ given by~\eqref{eq:scaledstate}, we first notice that
    \begin{equation*}
        \tilde\psi_i^{(0)} = \frac{m_i}{m_i + \Delta t m_i^{\sigma_a}} \pz_i > 0,
    \end{equation*}
    since $m_i>0$ and $\Delta t m_i^{\sigma_a}\geq 0$.
    Using the fact that $m_i^{\sigma_t} \geq m_i^{\sigma_a}$, we obtain the estimate
    \begin{equation*}
        \frac{\left|\tilde{\bm\psi}_i^{(1)}\right|}{ \tilde\psi_i^{(0)}}
        = \frac{m_i +\Delta t m_i^{\sigma_a}}{m_i + \Delta t m_i^{\sigma_t}}\frac{|\po_i|}{\pz_i}\leq \frac{|\po_i|}{\pz_i}<1.
    \end{equation*}
    Therefore, $\tilde u_i = (\tilde\psi_i^{(0)}, \tilde{\bm\psi}_i^{(1)})^\top\in\mathcal{R}_1$, as claimed in the lemma.
\end{proof}

The IDP property of the fully discrete implicit-explicit low-order scheme~\eqref{eq:fullLO} follows directly from Lemma~\ref{lem:scaledstate} and the previously established fact that $u_i^{\mathrm{SSP},S}\in\mathcal{R}_1$ under the CFL-like 
condition~\eqref{eq:CFL}.

\begin{remark}
    In order to use a generic implementation of SSP-RK methods in an existing code, such as the open-source \texttt{C++} finite element library MFEM~\cite{anderson2021,andrej2024,mfem}, we can rewrite the implicit-explicit Euler stages of our fully discrete low-order method as
    \begin{align*}
        u_i^{\mathrm{SSP}} &= \frac{m_i}{m_i + \Delta t m_i^{\sigma}} \left( u_i + \frac{\Delta t }{m_i}\left(\sum_{j\in\N_i^*}[ d_{ij} (u_j -u_i) - (\f_j -\f_i)\cdot \con_{ij}] + s_i\right)\right)\\
        & = u_i + \Delta t \left(\frac{1}{\Delta t} \left(\frac{m_i}{m_i + \Delta t m_i^{\sigma}} -1 \right) u_i + \frac{1}{m_i + \Delta t m_i^{\sigma}} \left(\sum_{j\in\N_i^*}[ d_{ij} (u_j -u_i) - (\f_j -\f_i)\cdot \con_{ij}]  +     s_i\right)\right).
    \end{align*}
   This is an update of the form $u^{n+1} = u^n +\Delta t g(u^n)$, which 
   reduces to the forward Euler stage \eqref{eq:LOparticlesource}
   if $m_i^{\sigma}=0$ because $\sigma_a = \sigma_s = 0$.
\end{remark}

\begin{remark}
    The theoretical results of this section can be extended to higher-order moment models derived from the LBE~\eqref{eq:Boltzmann}.
    The result of Lemma~\ref{lem:upm}, and thus of Theorem~\ref{thm:IDPbarstates}, can be adapted to any $M_N$, $N\geq 1$ model as long as it is equipped with a physical closure. 
    We conclude that the approach we used to derive the low-order IDP scheme for the $M_1$ model can  be applied to higher-order $M_N$ moment models similarly.
\end{remark}

\section{Monolithic convex limiting}
\label{sec:MCL}
The difference between the residuals of the semi-discrete CG formulation~\eqref{eq:stdCG} and of its low-order counterpart~\eqref{eq:sdLO} can be decomposed into an array of antidiffusive fluxes
\begin{equation}\label{eq:CGADF}
    f_{ij} = m_{ij} (\dot{u}_i - \dot u_j) + (d_{ij} + m_{ij}^{\sigma}) (u_i - u_j).
\end{equation}
The addition of $m_{ij} (\dot{u}_i - \dot u_j)$ and $m_{ij}^{\sigma} (u_i - u_j)$
on the right-hand side of \eqref{eq:sdLO} would correct the error due to mass lumping for the time derivative and reactive terms, respectively. The contribution of $d_{ij}(u_i- u_j)$ would offset the diffusive fluxes that appear on the right-hand side of~\eqref{eq:sdLO}.

To avoid solving the linear system~\eqref{eq:stdCG} and stabilize the
CG discretization as in \cite{kuzmin2020,kuzmin2023,lohmann2019},
we approximate the consistent nodal time derivative $\dot u_i$ by
\begin{equation*}
    \dot u_i^L = \frac{1}{m_i}\left(\sum_{j \in\N_i^*} \left[d_{ij} (u_j -u_i) - (\f_j-\f_i)\cdot \con_{ij}\right] - m_i^{\sigma}u_i + s_i\right)
\end{equation*}
and use the modified \emph{raw antidiffusive fluxes}  
\begin{equation}\label{eq:rawADF}
    f_{ij}^s = m_{ij} (\dot{u}^L_i - \dot u_j^L) + (d_{ij} + m_{ij}^{\sigma}) (u_i - u_j)
\end{equation}
to define the stabilized \emph{target scheme} 
\begin{equation}\label{target}
m_i\frac{\diff u_i}{\diff t} = \sum_{j \in\N_i^*} \left[d_{ij} (u_j -u_i) - (\f_j-\f_i)\cdot \con_{ij} + f_{ij}^s\right] - m_i^{\sigma}u_i + s_i.
\end{equation}
Owing to the skew symmetry property $f_{ji}^s = -f_{ij}^s$, the total mass
remains unchanged but low-order stabilization built into~\eqref{eq:sdLO} 
is replaced by high-order
background dissipation.
\medskip

Similarly to \eqref{eq:LOhomogene},
the spatial semi-discretization \eqref{target} can be written
in the bar state form
 \begin{equation}\label{target2}
   \begin{split}
  m_i\frac{\diff u_i}{\diff t}
  &=\sum_{j \in\N_i^*} [2d_{ij} (\bar u_{ij} -u_i)+f_{ij}^s] - m_i^{\sigma}u_i + s_i,\\
  &=\sum_{j \in\N_i^*} 2d_{ij} (\bar u_{ij}^H -u_i) - m_i^{\sigma}u_i + s_i,
  \end{split}
\end{equation}
where $$\bar u_{ij}^H =
\bar u_{ij} + \frac{f_{ij}^s}{2d_{ij}}.$$
The so-defined high-order bar states $\bar u_{ij}^H$
do not necessarily
belong to the admissible set $\mathcal{R}_1$.
Using the monolithic convex limiting framework proposed in~\cite{kuzmin2020}, we replace $\bar u_{ij}^H$ by
\begin{equation}\label{eq:limitedbarstates}
    \bar u_{ij}^* = \bar u_{ij} +\frac{f_{ij}^*}{2d_{ij}}.
\end{equation}
The construction of the limited flux $f_{ij}^*\approx f_{ij}^s$ is guided by three objectives:
\begin{enumerate}
    \item Suppress spurious oscillations and numerical instabilities.
    \item Ensure that the limited bar states~\eqref{eq:limitedbarstates} belong to $\mathcal{R}_1$.
    \item Preserve the skew symmetry property $f_{ij}^* = -f_{ij}^*$.
\end{enumerate}
As shown by our analysis in the previous section, the second requirement implies that each stage of the fully discrete flux-corrected scheme produces a realizable state
\begin{equation}\label{eq:fullydiscreteMCL}
    u_i^{\mathrm{SSP}} = \frac{m_i}{m_i + \Delta t m_i^{\tilde\sigma}}\left[ u_i + \frac{\Delta t}{m_i}\left(\sum_{j\in\N_i^*}2d_{ij} (\overline{u}_{ij}^* -u_i) + s_i\right)\right]\,\in\mathcal{R}_1
\end{equation}
under the CFL-like condition~\eqref{eq:CFL}. Numerical stability can be enhanced by imposing local bounds on individual components of $\bar u_{ij}^*$ or scalar functions thereof (cf. \cite{dobrev2018, hajduk2019, hajduk2021, kuzmin2020, kuzmin2023, moujaes2025}).

The investigations performed in~\cite{chidyagwai2018, dzanic2025} indicate that componentwise limiting is a good approach to enforcing numerical admissibility conditions for the $M_1$ model.
Let $\phi_i\in\{\pz_i, \psi_{i,1}^{(1)},\ldots,  \psi_{i,d}^{(1)}\}$ be a component of $u_i\in \mathcal{R}_1$ with corresponding low-order bar-state component $\bar\phi_{ij}$ and raw antidiffusive flux $f_{ij}^{\phi}$, $j\in\N_i^*$.
We formulate the inequality constraints
\begin{equation}\label{eq:localbp}
  \begin{split}
    \phi_i^{\min}\leq\bar\phi_{ij}^* &= \bar\phi_{ij} +\frac{f_{ij}^{\phi,*}}{2d_{ij}} \leq \phi_i^{\max},\\
      \phi_j^{\min}\leq\bar\phi_{ji}^* &= \bar\phi_{ji} -\frac{f_{ij}^{\phi,*}}{2d_{ij}} \leq \phi_j^{\max}
    \end{split}
\end{equation}
for $f_{ij}^{\phi,*}=-f_{ji}^{\phi,*}$
using the local bounds
\begin{equation}\label{eq:localbounds}
    \phi_i^{\max} = \max\left\{ \max_{j\in\N_i} \phi_j, \max_{j\in\N_i^*} \bar\phi_{ij} \right\},\quad
    \phi_i^{\min} = \min\left\{ \min_{j\in\N_i} \phi_j, \min_{j\in\N_i^*} \bar\phi_{ij} \right\}.
\end{equation}
The limiting conditions defined by \eqref{eq:localbp}
and \eqref{eq:localbounds} are feasible
because they hold for $f_{ij}^{\phi,*} = 0$.

Rearranging~\eqref{eq:localbp}, we find that the limited counterpart
$f_{ij}^{\phi,*}$ of $f_{ij}^{\phi}$ should satisfy
\begin{equation}\label{eq:fluxbounds}
  \begin{split}
    2d_{ij}\left(\phi_i^{\min} -\bar\phi_{ij}\right)\leq f_{ij}^{\phi,*}&\leq 2d_{ij}\left( \phi_i^{\max} -\bar\phi_{ij}\right),\\
    2d_{ij}\left(\phi_j^{\min}-\bar\phi_{ji}
    \right)\leq -f_{ij}^{\phi,*}&\leq 2d_{ij}
    \left(\phi_j^{\max}-\bar\phi_{ji} \right).
    \end{split}
\end{equation}
It is easy to verify that the limited antidiffusive fluxes defined
by~\cite{kuzmin2020, kuzmin2023}
\begin{equation}\label{eq:limitedfluxes}
    f_{ij}^{\phi,*}=\begin{cases}
        \min\left\{f_{ij}^\phi, 2d_{ij}\min \left\{\phi_i^{\max} -\bar\phi_{ij}, \bar\phi_{ji} -\phi_j^{\min}\right\}\right\}&\text{if }f_{ij}^\phi>0,\\
        \max\left\{f_{ij}^\phi, 2d_{ij}\max \left\{\phi_i^{\min} -\bar\phi_{ij}, \bar\phi_{ji} -\phi_j^{\max}\right\}\right\}&\text{otherwise}\\
    \end{cases}
\end{equation}
are skew-symmetric and
satisfy the local maximum principles
\eqref{eq:localbp} for individual components
of $\bar u_{ij}^*$.

In addition to strong numerical stability, the use of
\eqref{eq:limitedfluxes} ensures that $\bar\psi^{(0),*}_{ij}\geq\psi_i^{(0),\min}>0$.
However, the flux-corrected bar state may still violate
the realizable velocity constraint $|\po|<\pz$.
We enforce this constraint in a second limiting step
by adapting the IDP fix designed to ensure positivity
preservation for the pressure (internal energy)
of the compressible Euler equations~\cite{kuzmin2020, kuzmin2023}.

Let $f_{ij}^* = (f_{ij}^{*(0)}, \f_{ij}^{*(1)})^\top$ be a
limited antidiffusive flux whose individual components are
given by~\eqref{eq:limitedfluxes}. We define the
final, physically admissible antidiffusive flux
\begin{equation*}
    f_{ij}^{\mathrm{IDP}} = \alpha_{ij}^{\mathrm{IDP}} f_{ij}^*
\end{equation*}
using a scalar correction factor 
$\alpha_{ij}^{\mathrm{IDP}}\in[0,1]$ such that
\begin{equation}\label{eq:HOIDPbarstats}
    \overline{u}_{ij}^{\mathrm{IDP}} = \overline{u}_{ij} +\frac{\alpha_{ij}^{\mathrm{IDP}} f_{ij}^*}{2d_{ij}} \in\mathcal{R}_1.
\end{equation}
The positivity of the particle density is guaranteed for any
$\alpha_{ij}^{\mathrm{IDP}}\in[0,1]$ because it was enforced
in the componentwise limiting step. The
realizable velocity constraint can be formulated as
\begin{equation*}
    \left|\bar{\bm\psi}^{(1)}_{ij} + \frac{\alpha_{ij}^{\mathrm{IDP}} \f_{ij}^{*(1)}}{2d_{ij}}\right|^2 
    < \left(\bar\psi^{(0)}_{ij} + \frac{\alpha_{ij}^{\mathrm{IDP}} f_{ij}^{*(0)}}{2d_{ij}}\right)^2.
\end{equation*}
This inequality is equivalent to 
\begin{equation}\label{eq:pij<qij}
    P_{ij}(\alpha_{ij}^{\mathrm{IDP}}) < Q_{ij},
\end{equation}
where 
\begin{equation*}
    P_{ij}(\alpha) = \left(\left|\f_{ij}^{*(1)}\right|^2 - \left(f_{ij}^{*(0)}\right)^2 \right)\alpha^2 
    + 4 d_{ij} \left(\bar{\bm\psi}^{(1),*}_{ij}\cdot \f_{ij}^{*(1)} -  \bar\psi^{(0),*}_{ij}f_{ij}^{*(0)}\right)\alpha,
\end{equation*}
\begin{equation*}
    Q_{ij} = (2d_{ij})^2\left(
    \left(\bar\psi^{(0)}_{ij}\right)^2 - \left|\bar{\bm\psi}^{(1)}_{ij}\right|^2\right) > 0.
\end{equation*}
The positivity of $Q_{ij}$ follows from the IDP property of the low-order bar states. It follows that \eqref{eq:pij<qij} holds for the trivial choice $\alpha_{ij}^{\mathrm{IDP}} = 0$. Thus, the additional constraint \eqref{eq:pij<qij} is feasible.

Using the
fact that $\alpha^2 \leq \alpha$ for all $\alpha\in[0,1]$, we find that
$P_{ij} (\alpha)\leq \alpha R_{ij}$ for all $\alpha\in[0,1]$ and
\begin{equation*}
    R_{ij} = \max\left\{  0, \left|\f_{ij}^{*(1)}\right|^2 - \left(f_{ij}^{*(0)}\right)^2  \right\} + 4 d_{ij} \left(\bar{\bm\psi}^{(1),*}_{ij}\cdot \f_{ij}^{*(1)} -  \bar\psi^{(0),*}_{ij}f_{ij}^{*(0)}\right).
\end{equation*}
Let $\tilde{Q}_{ij} = (1-\varepsilon) Q_{ij}>0$ with $\varepsilon = 10^{-15}$.
Then the application of
\begin{equation*}
    \alpha_{ij}^{\mathrm{IDP}} = \begin{cases}
        \min\left\{ \frac{\tilde Q_{ij}}{R_{ij}} , \frac{\tilde Q_{ji}}{R_{ji}} \right\} &\text{if } R_{ij} > \tilde Q_{ij}, R_{ji} >\tilde Q_{ji},\\
        \frac{\tilde Q_{ij}}{R_{ij}} &\text{if }R_{ij} > \tilde Q_{ij}, R_{ji} \leq  \tilde Q_{ji},\\
        \frac{\tilde Q_{ji}}{R_{ji}} &\text{if }R_{ij} \leq  \tilde Q_{ij}, R_{ji} >  \tilde Q_{ji},\\
        1 &\text{otherwise}
    \end{cases}
\end{equation*}
to all components of the prelimited antidiffusive flux $f_{ij}^* = (f_{ij}^{*(0)}, \f_{ij}^{*(1)})^\top$ ensures that
\begin{equation*}
    P_{ij}(\alpha_{ij}^{\mathrm{IDP}}) \leq \alpha_{ij}^{\mathrm{IDP}} R_{ij} \leq \tilde Q_{ij} < Q_{ij} \quad \text{and} \quad P_{ji}(\alpha_{ij}^{\mathrm{IDP}}) \leq \alpha_{ij}^{\mathrm{IDP}} R_{ji} \leq \tilde Q_{ji}< Q_{ji}. 
\end{equation*}
Therefore, $\overline{u}_{ij}^{\mathrm{IDP}}\in\mathcal{R}_1$ whenever $\overline{u}_{ij}\in\mathcal{R}_1$.
Substituting $\overline{u}_{ij}^{\mathrm{IDP}}$ for  $\overline{u}_{ij}$
in~\eqref{eq:fullydiscreteMCL}, we obtain a numerically stable and physically
admissible high-order IDP discretization of the $M_1$ model.

\section{Numerical examples}
\label{sec:examples}

To evaluate the proposed limiting strategy and compare it with approaches employed in the literature, we apply our realizability-preserving MCL scheme to representative test problems. For temporal discretization, we use Heun's scheme, a second-order explicit SSP-RK method. In the inhomogeneous case, lumped reactive terms are treated implicitly, as in the low-order update~\eqref{eq:fullLO}. Steady-state computations are performed using pseudo-time stepping with a single implicit-explicit Euler stage. In view of condition \eqref{eq:CFL}, the time step $\Delta t$ is determined using the formula~\cite{guermond2016, kuzmin2020, kuzmin2023} 
\begin{equation}\label{eq:timestepsize}
 \Delta t
  = \frac{\mathrm{CFL}}{\max_{i\in\{1,\ldots,N_h\}}\frac{2}{m_i}\sum_{j\in\N^*_i}d_{ij}},
\end{equation}
where $\mathrm{CFL} \leq 1$ is a given threshold.
This choice of $\Delta t$ guarantees realizability, as shown by our analysis in Sections~\ref{sec:LO} and~\ref{sec:MCL}.
Note that the time stepping based on~\eqref{eq:timestepsize} is independent of the solution and its evolution. 
Thus, the time step needs to be evaluated just once in a preprocessing step.

The implementation of MCL that we test in our numerical experiments is based on the open-source \texttt{C++} finite element library MFEM~\cite{anderson2021,andrej2024,mfem}. The results are visualized in Paraview~\cite{ayachit2015}.

\subsection{Line source}
To test the shock capturing capabilities of our numerical scheme, we consider the \emph{line source} benchmark for the time-dependent $M_1$ model~\cite{brunner2005}.
This experiment corresponds to a Green function problem, in which an isotropic, instantaneous pulse of radiation is emitted from a line source at the center of the two-dimensional domain $\D = (-0.5,0.5)^2$.  
The exact solution is radially symmetric and features a steep shock-like front, which makes it a challenging test for numerical methods.

While the original setup in~\cite{brunner2005} models radiative transfer in a purely scattering medium, we adopt a vacuum configuration ($\sigma_s = \sigma_a = 0$) for a better comparison with the limiting strategies that were applied to the $M_1$ model in~\cite{chidyagwai2018}.
Furthermore, we assume that no particles are created and set $q = 0$.
The initial condition is given by a smooth approximation of a Dirac delta distribution
\begin{equation*}
\pz(0,x, y) = \max\left( \exp\left( - 10\frac{x^2 + y^2}{\theta^2} \right), 10^{-4} \right), \quad \po(0,x, y) = 0,
\end{equation*}
where $\theta = 0.02$.
Since the wave does not reach the boundary during the simulation with the final time $t_{\mathrm{final}} = 0.45$, no boundary conditions are required.

\begin{figure}[h!]
	\centering
	\begin{subfigure}[c]{0.45\textwidth}
		\includegraphics[trim={14cm 0cm 10.5cm 0cmm},clip, width = 0.95\textwidth]{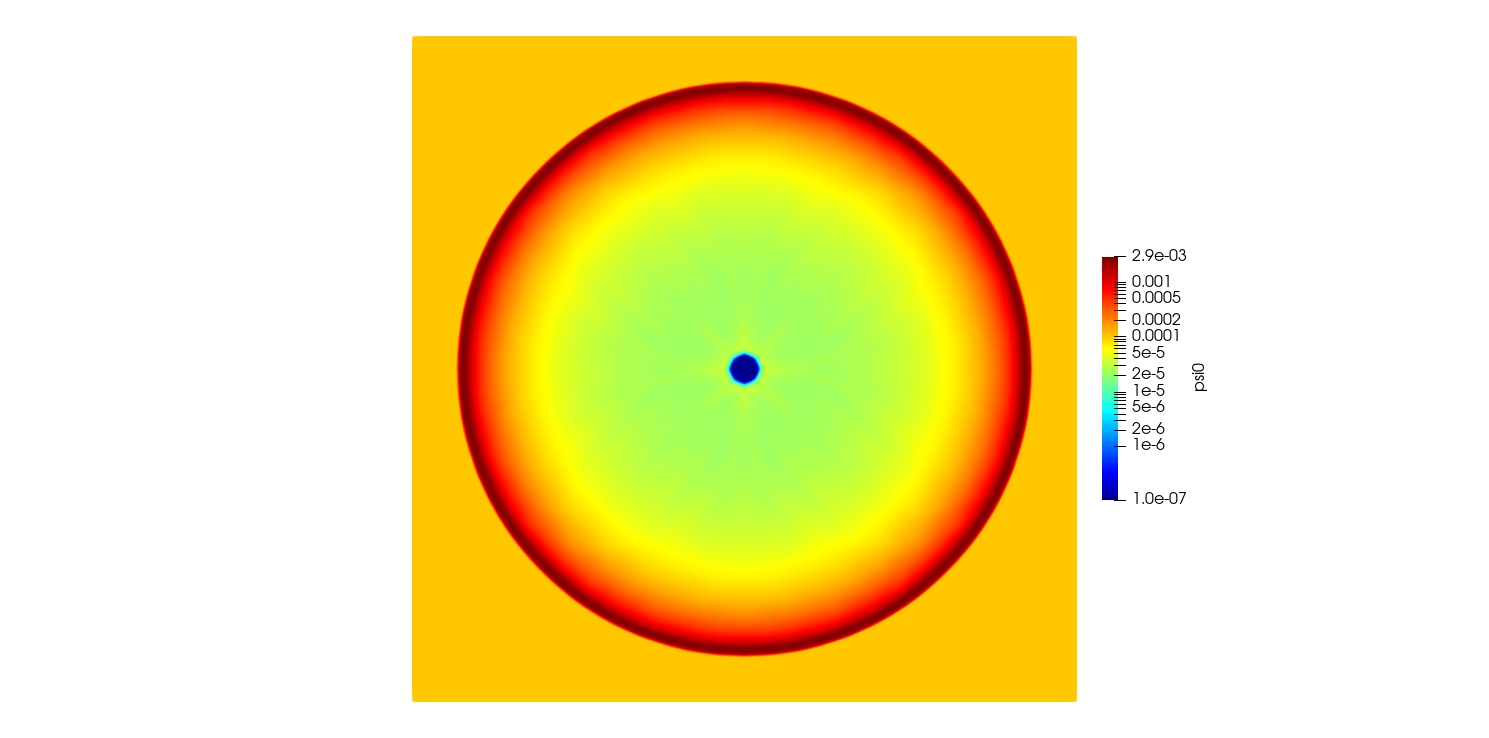}
		\subcaption{$\pz$ (logarithmic scale)}
	\end{subfigure}
	\begin{subfigure}[c]{0.45\textwidth}
		\includegraphics[trim={14cm 0.5cm 10.6cm 0cmm},clip, width = 0.95\textwidth]{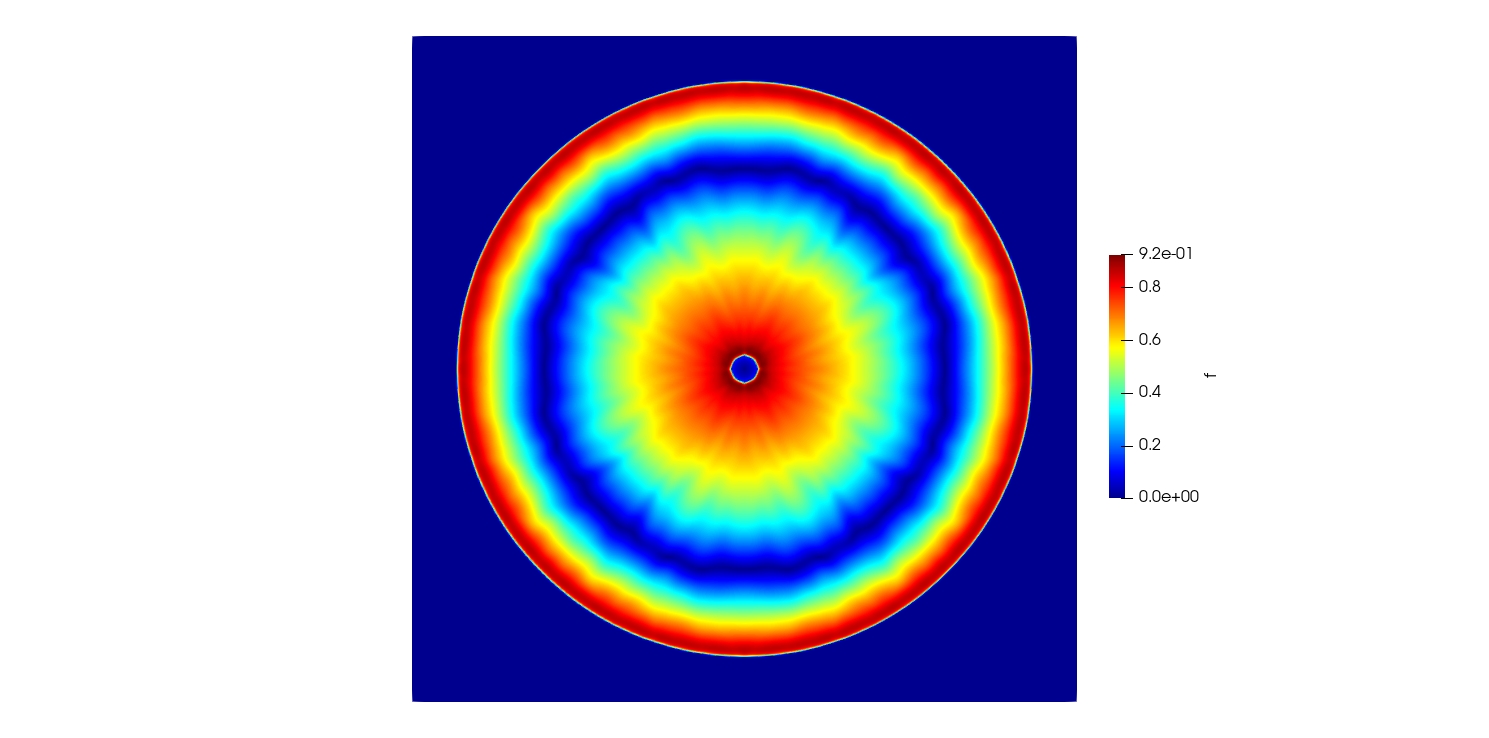}
		\subcaption{$f =\frac{|\po|}{\pz}$}
	\end{subfigure}
	\caption{Line source simulation results at $t = 0.45$ computed with the MCL scheme using a uniform rectangular mesh with $N_h = 512^2$ nodes per component and $\mathrm{CFL} = 0.5$.}
    \label{fig:linesource}
\end{figure}

This problem is particularly sensitive to numerical artifacts, which can lead to a loss of symmetry or a violation of realizability (see, e.g.,~\cite{chidyagwai2018}).  
As shown in Fig.~\ref{fig:linesource}, the proposed MCL scheme resolves the shock in a sharp and stable manner.  The rotational symmetry is preserved and physical admissibility is maintained throughout the simulation without introducing excessive numerical diffusion.

\subsection{Flash test}

Another homogeneous benchmark is the \emph{flash} test~\cite{kanno2013}.
This experiment simulates a bulk of mass moving from the center of the domain $\D = (-10,10)^2$ to the right boundary.
Let $\D_{\frac{1}{2}} = \{(x,y)\in \R^2: \sqrt{x^2 +y^2}\leq \frac{1}{2}\}$ be the disc centered at the origin with radius $r=\frac{1}{2}$.
The initial condition
\begin{equation*}
    u(x,y) = (\pz, \psi^{(1)}_x, \psi^{(1)}_y)^\top = \begin{cases}
        (1, 0.9, 0)^\top &\text{if } (x,y)\in \D_{\frac{1}{2}},\\
        (10^{-10}, 0, 0)^\top &\text{otherwise}
    \end{cases}
\end{equation*}
is close to the boundary of the realizable set $\mathcal R_1$, since $f = \frac{|\po|}{\pz} = 0.9$ on $\D_\frac{1}{2}$.

As mentioned above, we consider the homogeneous $M_1$ system in this test, i.e.,
\begin{equation*}
    q = 0,\quad \sigma_a = \sigma_s = 0
\end{equation*}
in the whole domain.
We run the simulation until $t_{\mathrm{final}} = 6$.
Since the moving mass does not reach the boundary
for $t\le t_{\mathrm{final}}$, no boundary conditions need to be prescribed.

\begin{figure}[h!]
	\centering
	\begin{subfigure}[c]{0.45\textwidth}
		\includegraphics[trim={14cm 0cm 10.5cm 0cmm},clip, width = 0.95\textwidth]{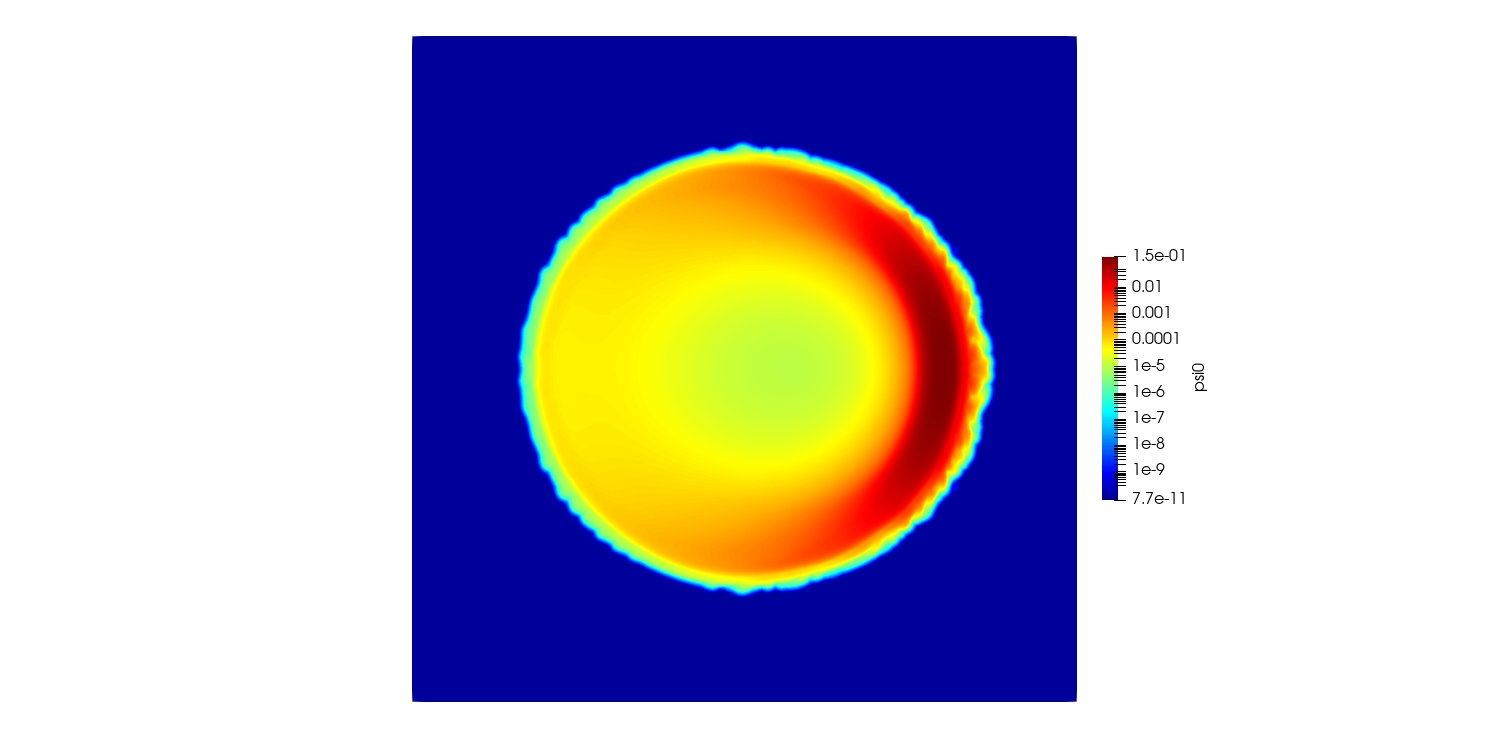}
		\subcaption{$\pz$ (logarithmic scale)}
        \label{fig:flashtest_psi0}
	\end{subfigure}
	\begin{subfigure}[c]{0.45\textwidth}
		\includegraphics[trim={14cm 0cm 10.5cm 0cmm},clip, width = 0.95\textwidth]{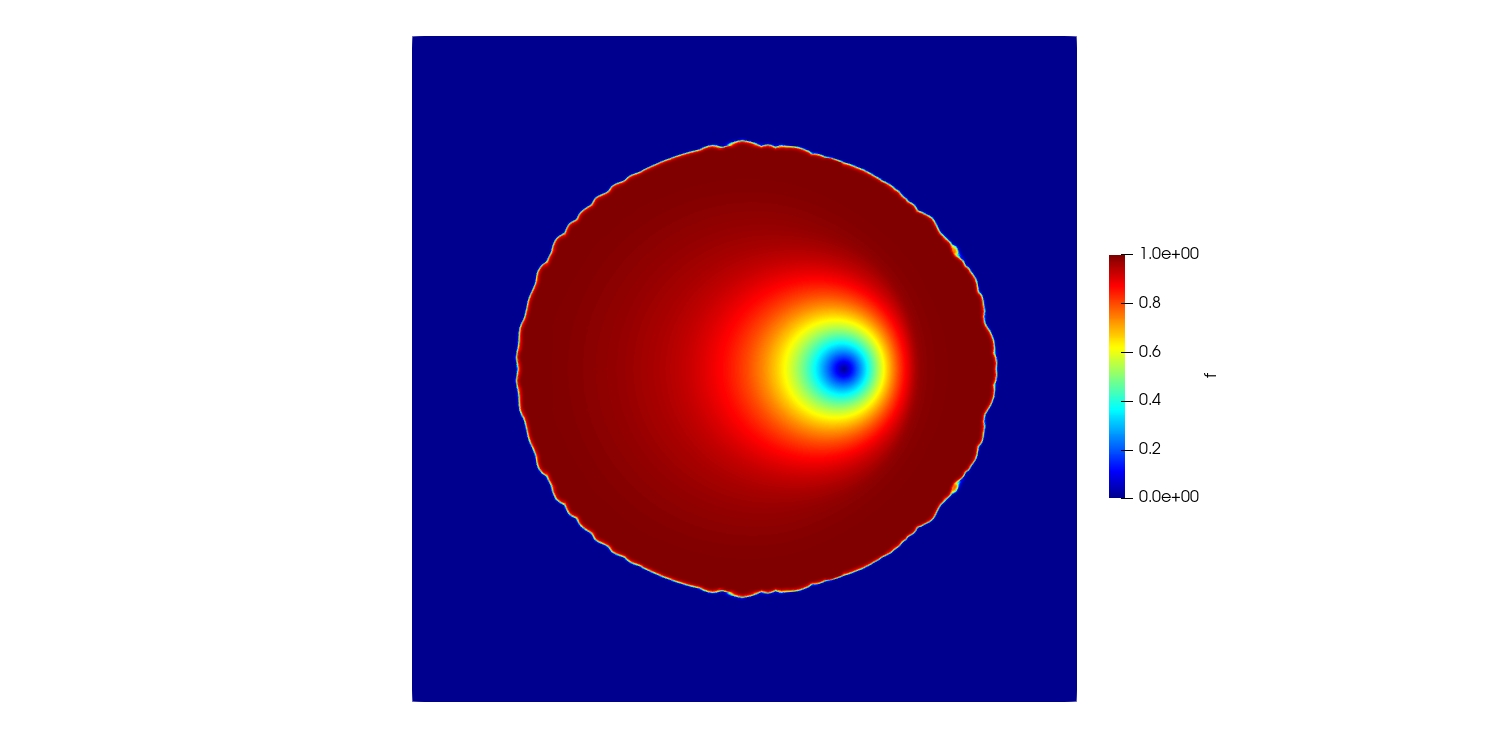}
		\subcaption{$f =\frac{|\po|}{\pz}$}
        \label{fig:flashtest_f}
	\end{subfigure}
	\caption{Flash simulation results at $t = 6$ computed with the MCL scheme using a uniform rectangular mesh with $N_h = 512^2$ nodes per component and $\mathrm{CFL} = 0.5$.}
    \label{fig:flashtest}
\end{figure}

Figure~\ref{fig:flashtest} shows the numerical results for this benchmark.
 The numerical solutions
 displayed in Fig.~\ref{fig:flashtest_f} are very close to the
boundary of $\mathcal R_1$ in large parts of the computational domain. 
In fact, $f\leq 1 - 2.32\times10^{-9}$, which makes this problem very difficult and emphasizes the importance of the IDP fix proposed in Section~\ref{sec:MCL}.
No unacceptable states were detected throughout the computation.

\subsection{Homogeneous disk}

Next, we consider the \emph{homogeneous disk} test~\cite{chidyagwai2018}, in which  
 a static homogeneous radiating region is embedded in vacuum.
 We define the computational domain as $\D = (-5,5)^2$ and evolve
 the moments up to the final time $t_{\mathrm{final}} = 3$.
Let $\D_1 = \{(x, y) \in \R^2 : x^2 + y^2 \leq 1 \}$ denote the unit disk.  
The material parameters and the source term of the $M_1$ model are given by
\begin{align}
\sigma_a(x, y) &= \begin{cases} 10 & \text{if } (x,y) \in \D_1, \\ 0 & \text{otherwise}, \end{cases}\quad \sigma_s(x, y) = 0, \\
q^{(0)}(x, y) &= \begin{cases} 1 & \text{if } (x,y) \in \D_1, \\ 0 & \text{otherwise}, \end{cases}\quad \mathbf{q}^{(1)}(x, y) = 0,
\end{align}
respectively.
The discontinuity of material parameters on the boundary of the unit disk $\D_1$ makes this problem numerically challenging. 
The initial conditions 
\begin{equation*}
\pz(x,y,0) = 10^{-10}, \quad \po(x,y,0) = 0
\end{equation*}
correspond to background radiation with low constant intensity.
Again, since the wave originating from the source does not reach the boundary during the simulation run, no boundary conditions are needed.
As in the line source problem, the solution is expected to be radially symmetric.

\begin{figure}[h!]
	\centering
	\begin{subfigure}[c]{0.45\textwidth}
		\includegraphics[trim={14cm 0cm 11cm 0cm},clip, width = 0.95\textwidth]{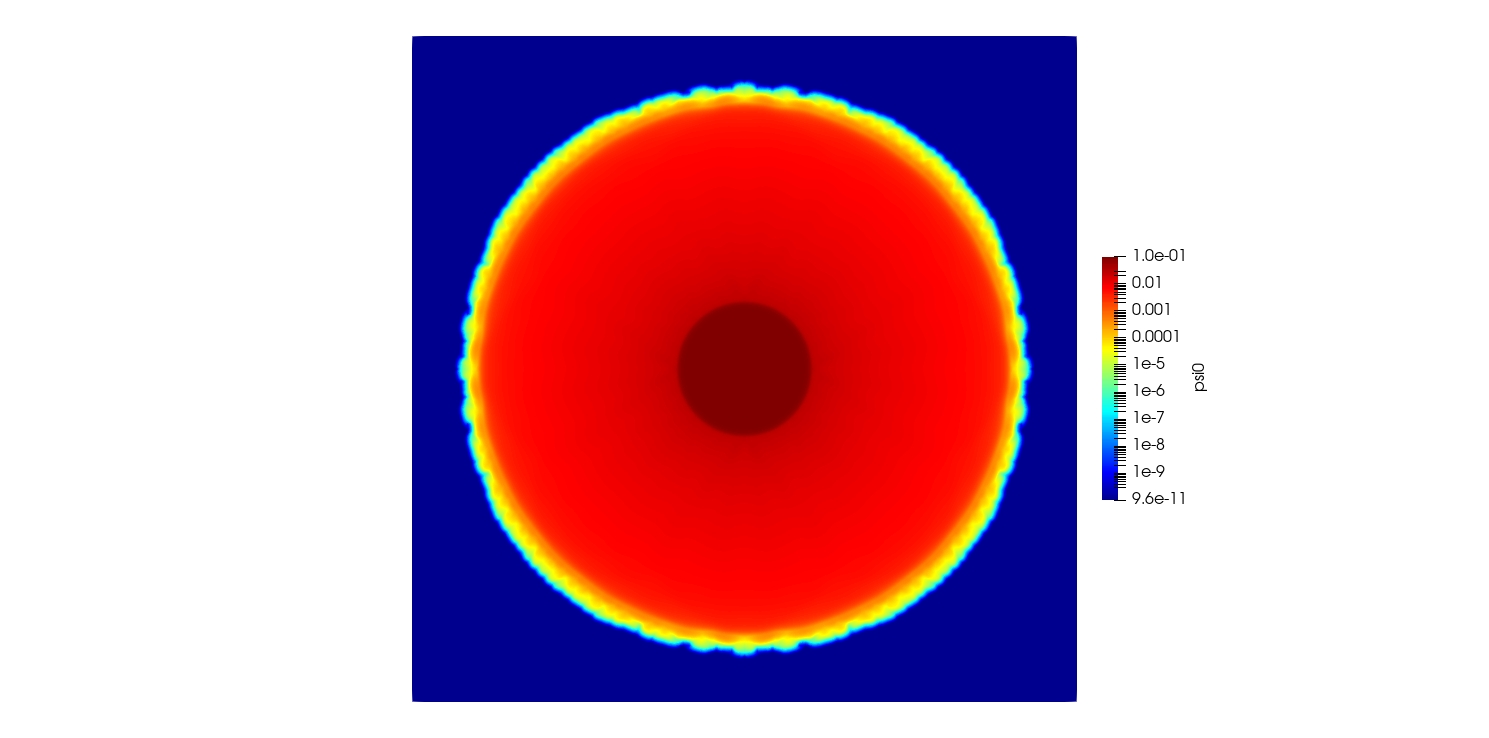}
		\subcaption{$\pz$ (logarithmic scale)}
	\end{subfigure}
	\begin{subfigure}[c]{0.45\textwidth}
		\includegraphics[trim={14cm 0.5cm 10.6cm 0cm},clip, width = 0.95\textwidth]{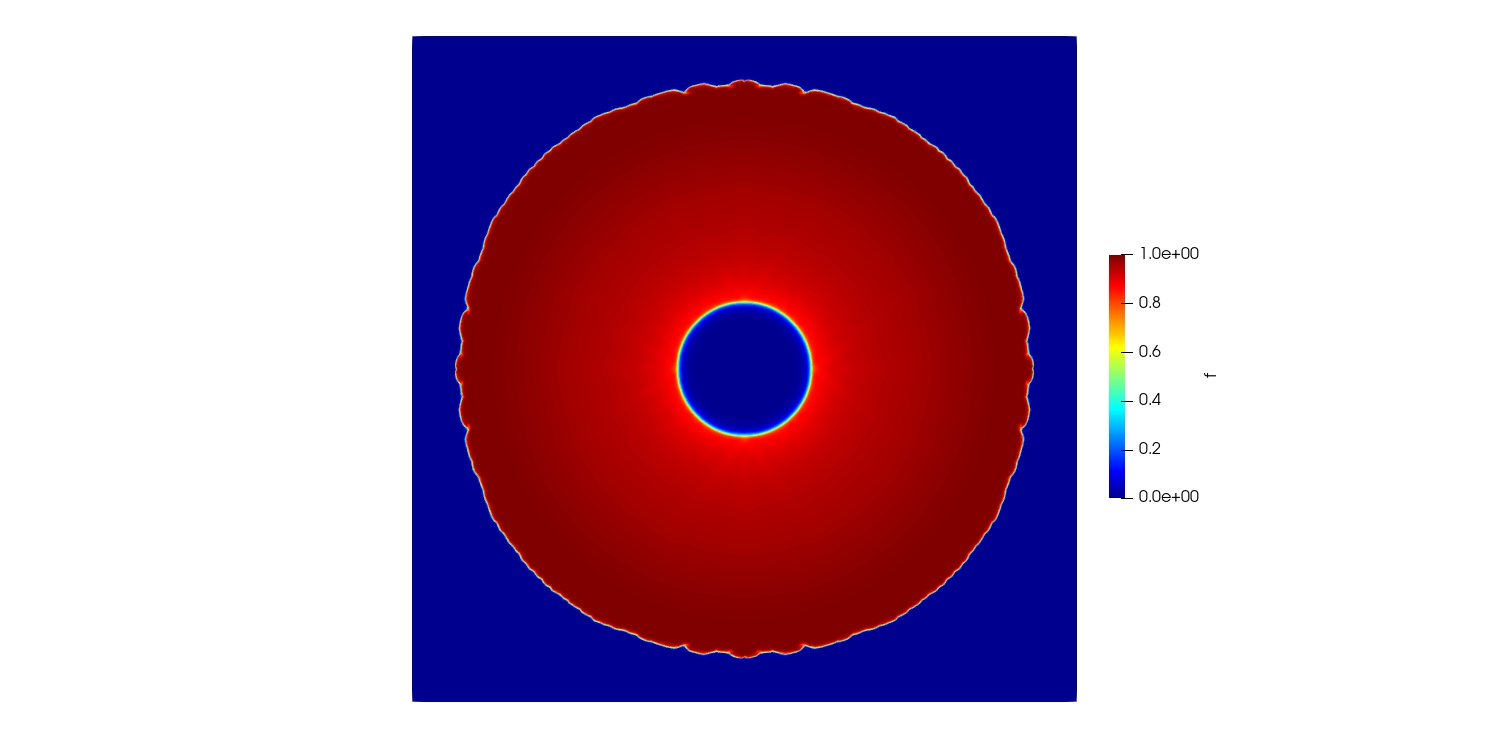}
		\subcaption{$f =\frac{|\po|}{\pz}$}
	\end{subfigure}
	\caption{Homogeneous disk simulation results at $t = 3.0$ computed with the MCL scheme using a uniform rectangular mesh with $N_h = 512^2$ nodes per component and $\mathrm{CFL} = 0.5$.}
    \label{fig:disk}
\end{figure}

The MCL results presented in Fig.~\ref{fig:disk} are nonoscillatory, realizable, and exhibit high resolution of the discontinuities caused by the abrupt change of the forcing terms across the boundary of the disc $\mathcal D_1$. Minor deviations from the exact circular shape of the outer interface can be attributed to componentwise limiting and/or lack of high-order nonlinear stabilization in the target scheme.

\subsection{Lattice problem}
Another challenging benchmark with discontinuous material parameters is the \emph{lattice} problem introduced in~\cite{brunner2002}. The computational domain $\D = (0,7)^2$ is filled with a scattering background medium and an array of highly absorbing materials that are arranged in a checkerboard pattern.
\begin{figure}[h!]
	\centering
		\includegraphics[trim={14cm 0cm 14cm 0cmm},clip, width = 0.4\textwidth]{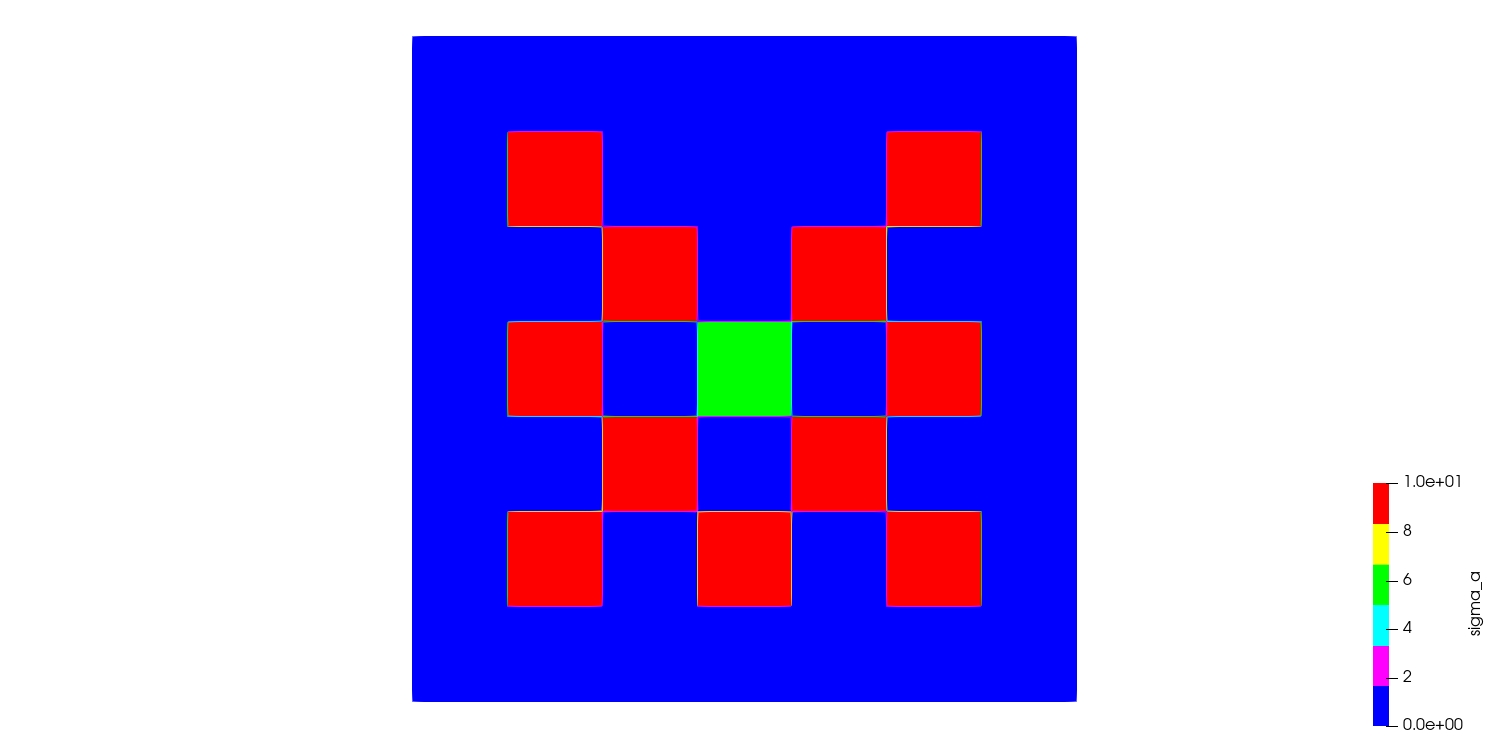}
		\caption{Lattice problem setup: absorbing region $\D_a$ is shown in red; sources are concentrated in the green region.}
        \label{fig:Lattice_domain}
\end{figure}
We illustrate the structural distribution of material properties in Fig.~\ref{fig:Lattice_domain}, where we plot the absorbing region
\begin{equation}\label{eq:D_a}
    \begin{split}
        \D_{a} =\ & \bigl( [1,2] \cup [5,6] \bigr)\times \bigl( [1,2] \cup [3,4] \cup [5,6] \bigr) \\
        \cup\ & \bigl([2,3] \cup [4,5]\bigr) \times \bigl([2,3] \cup [4,5] \bigr)  \\
        \cup\ & [3,4] \times [1,2]
    \end{split}
\end{equation}
in red.
Using~\eqref{eq:D_a}, we define the absorption and scattering parameters as
\begin{equation*}
    \sigma_a(x,y) = \begin{cases}
        10 &\text{if } (x,y) \in \D_{a},\\
        0 & \text{otherwise},
    \end{cases}
    \quad \sigma_s(x,y) =\begin{cases}
        1 & \text{if } (x,y) \in \D\setminus \D_{a},\\
        0 & \text{otherwise}.
    \end{cases}
\end{equation*}
The particle source term
\begin{equation}\label{eq:lattice_source1}
        \quad q^{(0)}(x,y) =\begin{cases}
            1 & \text{if } (x,y) \in [3,4]\times[3,4],\\
            0 & \text{otherwise},
        \end{cases}\qquad \mathbf{q}^{(1)}\equiv0
\end{equation}
of the original benchmark is isotropic.
To further demonstrate the realizability of the proposed MCL scheme for general sources $q\in\overline{\mathcal{R}_1}$, we perform a second test with the
 anisotropic particle source
\begin{equation}\label{eq:lattice_source2}
        \quad q^{(0)}(x,y) =\begin{cases}
            1 & \text{if } (x,y) \in [3,4]\times[3,4],\\
            0 & \text{otherwise},
        \end{cases}\qquad \mathbf{q}^{(1)}(x,y)=\begin{cases}
             (0, -1)^\top & \text{if } (x,y) \in [3,4]\times[3,4],\\
            (0,0)^\top & \text{otherwise}.
        \end{cases}
\end{equation}
Note that the state defined by~\eqref{eq:lattice_source2} lies on the boundary of the realizable set $\mathcal R_1$.
It corresponds to the moments of an angular delta distribution given by $Q(\OO)=\delta(\OO+\mathbf{e}_2)$, where $\mathbf{e}_2$ is the unit vector in the positive $y$-direction.
This setup represents a perfectly collimated beam traveling downward.

We prescribe a do-nothing boundary condition at the outlet of $\mathcal D$.
Using the
initial condition 
\begin{equation}\label{eq:Lattice_init}
\pz(x,y,0) = 10^{-10}, \quad \po(x,y,0) = 0,
\end{equation}
we perform transient and steady-state computations for sources defined by~\eqref{eq:lattice_source1} and~\eqref{eq:lattice_source2}.
In the transient scenarios, the moments are evolved up to the final time $t_{\mathrm{final}} = 3.2$. For steady-state computations, we use single-stage pseudo-time stepping and the residual-based stopping criterion
\begin{equation*}
    \|r_h\|_{L^2(\D)} \leq 10^{-8},
\end{equation*}
where
\begin{equation*}
   r_h = \sum_{i=1}^{N_h} r_i \varphi_i,\qquad
   r_i =\frac{1}{m_i}\left(\tilde{b}_i(u_i, \hat u_i) +\sum_{j \in\N_i^*}2d_{ij} (\bar u_{ij}^{\mathrm{IDP}} -u_i) - m_i^{\sigma}u_i + s_i \right) 
\end{equation*}
is the finite element function corresponding to the time derivative of the MCL solution.

\begin{figure}[h!]
	\centering
	\begin{subfigure}[c]{0.45\textwidth}
		\includegraphics[trim={14cm 0cm 10.67cm 0cm},clip, width = 0.95\textwidth]{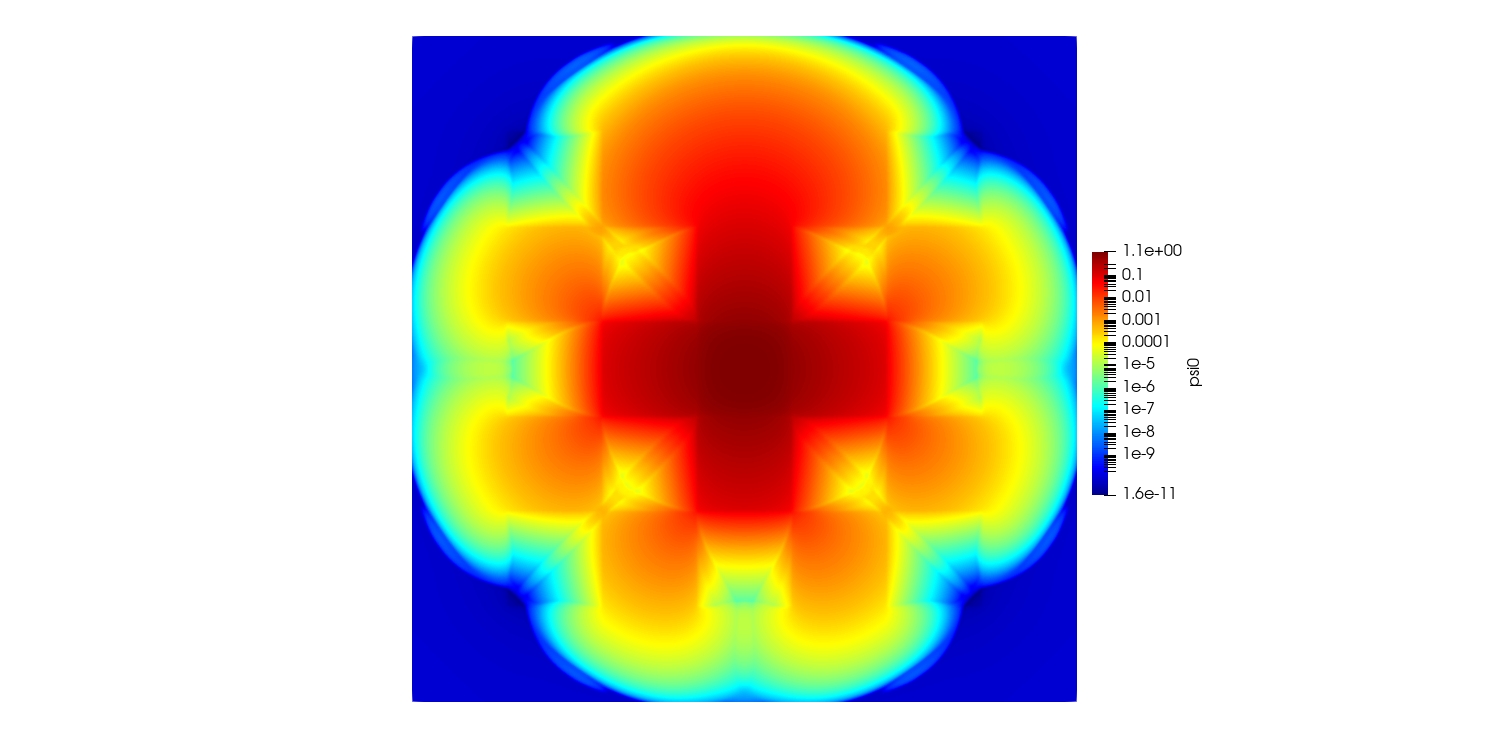}
		\subcaption{$\pz$ (logarithmic scale), isotropic source~\eqref{eq:lattice_source1}}
	\end{subfigure}
	\begin{subfigure}[c]{0.45\textwidth}
		\includegraphics[trim={14cm 0cm 10.67cm 0cm},clip, width = 0.95\textwidth]{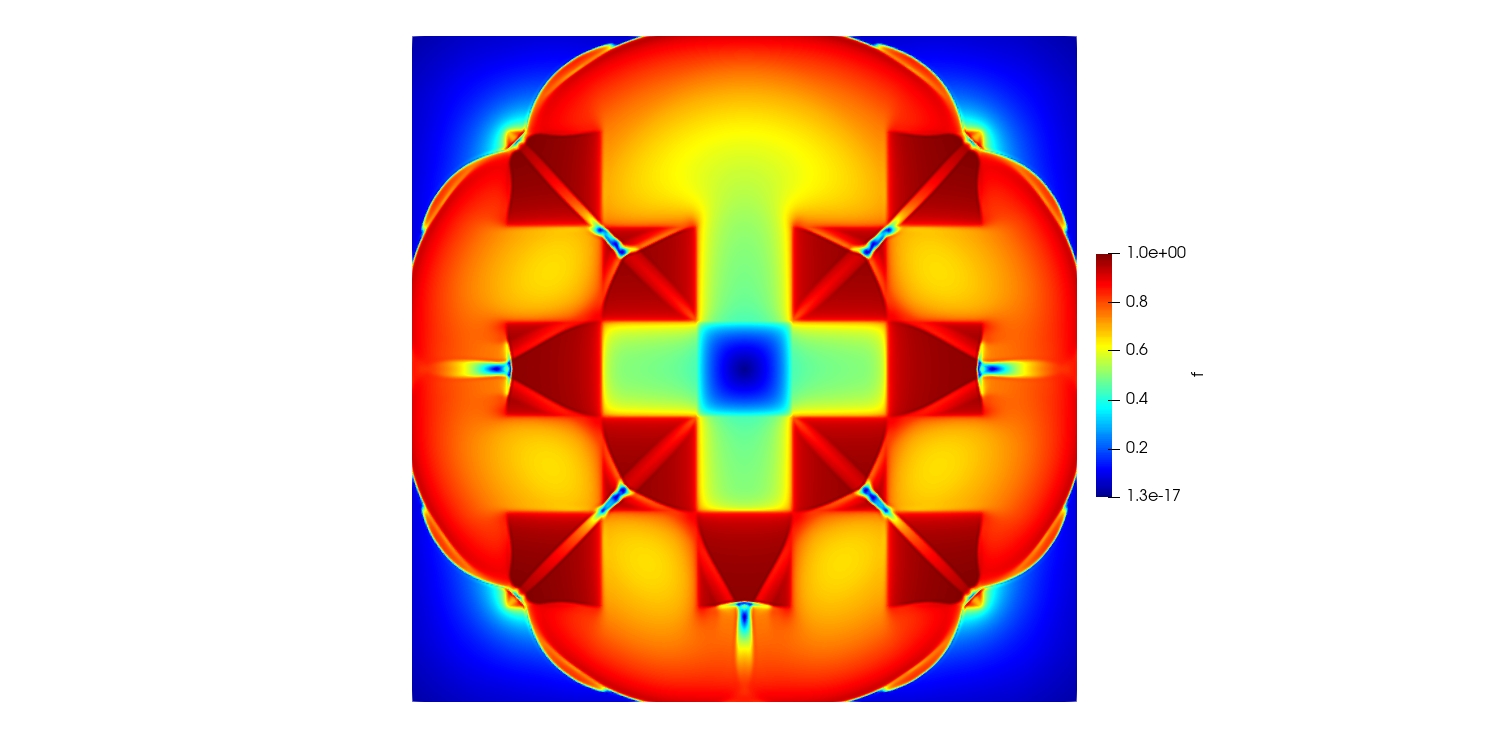}
		\subcaption{$f =\frac{|\po|}{\pz}$, isotropic source~\eqref{eq:lattice_source1}}
	\end{subfigure}
        \begin{subfigure}[c]{0.45\textwidth}
		\includegraphics[trim={14cm 0cm 10.67cm 0cm},clip, width = 0.95\textwidth]{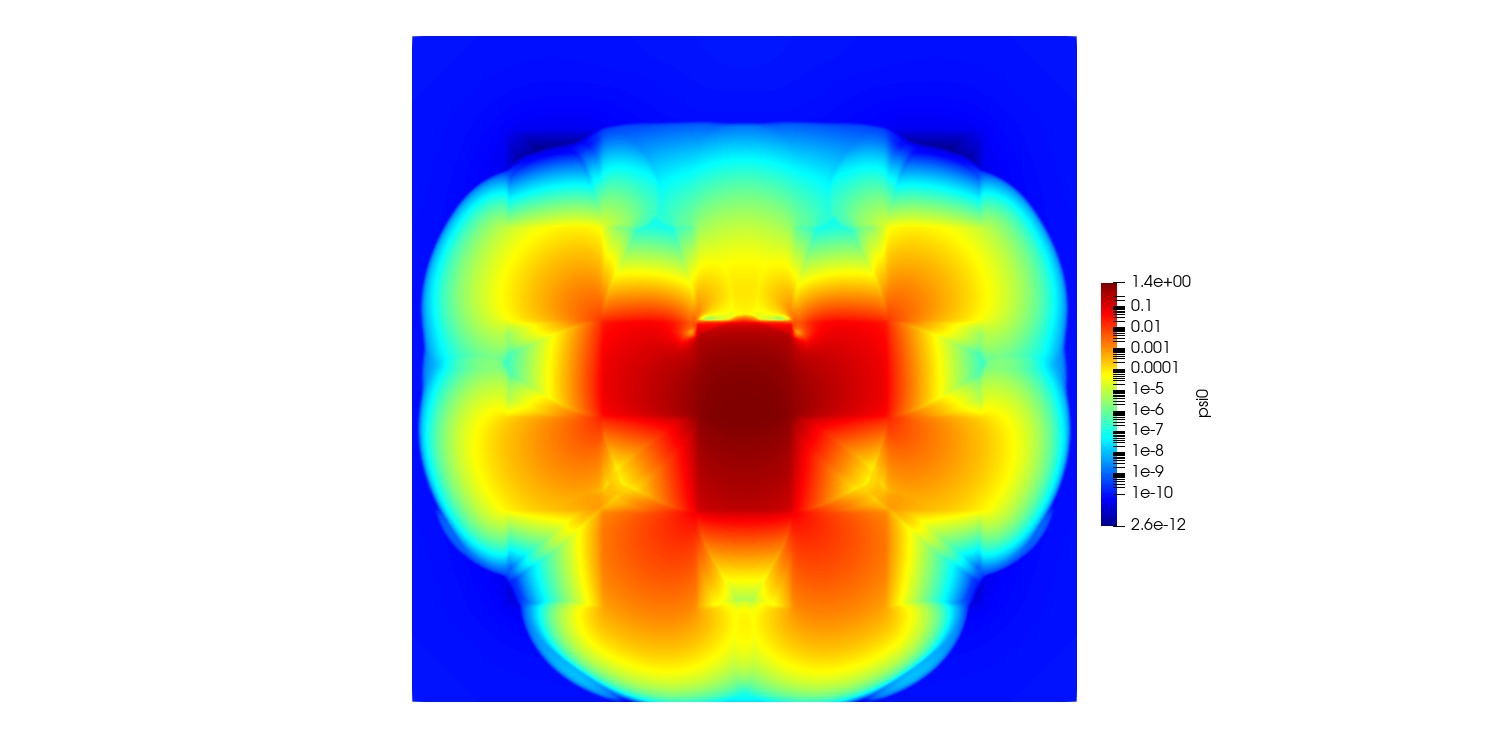}
		\subcaption{$\pz$ (logarithmic scale), anisotropic source~\eqref{eq:lattice_source2}}
	\end{subfigure}
	\begin{subfigure}[c]{0.45\textwidth}
		\includegraphics[trim={14cm 0cm 10.67cm 0cm},clip, width = 0.95\textwidth]{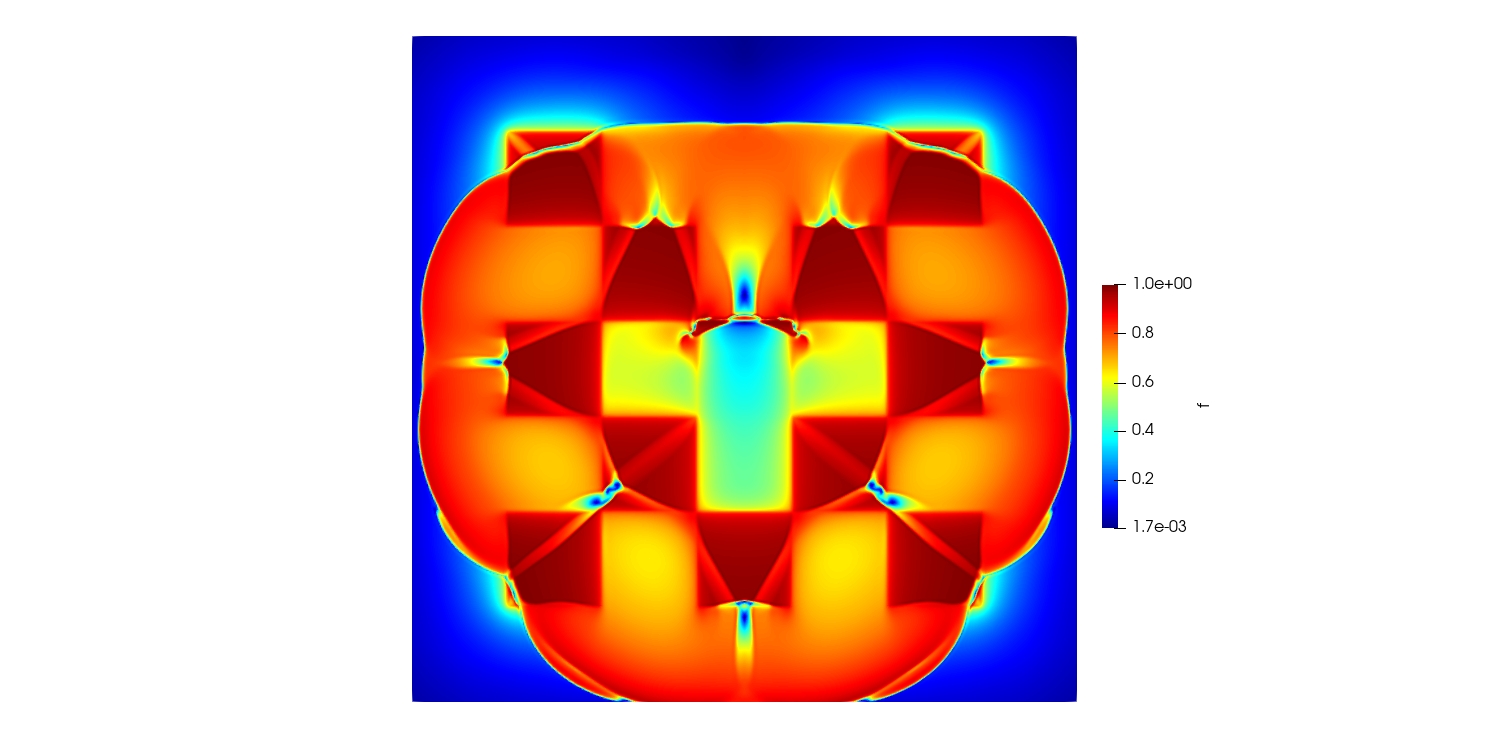}
		\subcaption{$f =\frac{|\po|}{\pz}$, anisotropic source~\eqref{eq:lattice_source2}}
	\end{subfigure}
	\caption{Transient lattice simulation results at $t = 3.2$ computed with the MCL scheme using a uniform rectangular mesh with $N_h = 512^2$ nodes per component and $\mathrm{CFL} = 0.5$. In the test (a,b), the source term was defined by~\eqref{eq:lattice_source1}, while ~\eqref{eq:lattice_source2} was used in the test (c,d).}
    \label{fig:transientLattice}
\end{figure}

\begin{figure}[h!]
	\centering
	\begin{subfigure}[c]{0.45\textwidth}
		\includegraphics[trim={14cm 0cm 10.67cm 0cm},clip, width = 0.95\textwidth]{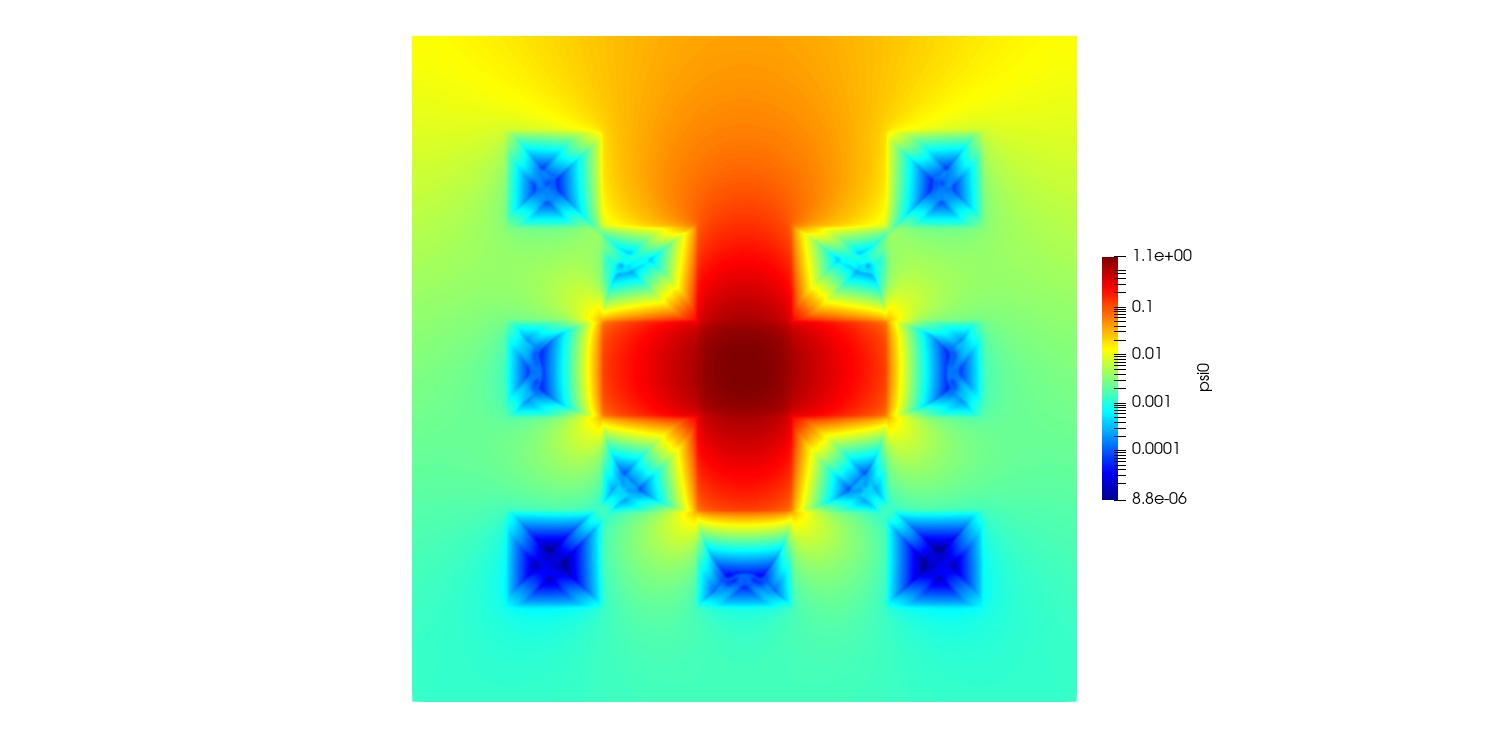}
		\subcaption{$\pz$ (logarithmic scale), isotropic source~\eqref{eq:lattice_source1}}
	\end{subfigure}
	\begin{subfigure}[c]{0.45\textwidth}
		\includegraphics[trim={14cm 0cm 10.67cm 0cm},clip, width = 0.95\textwidth]{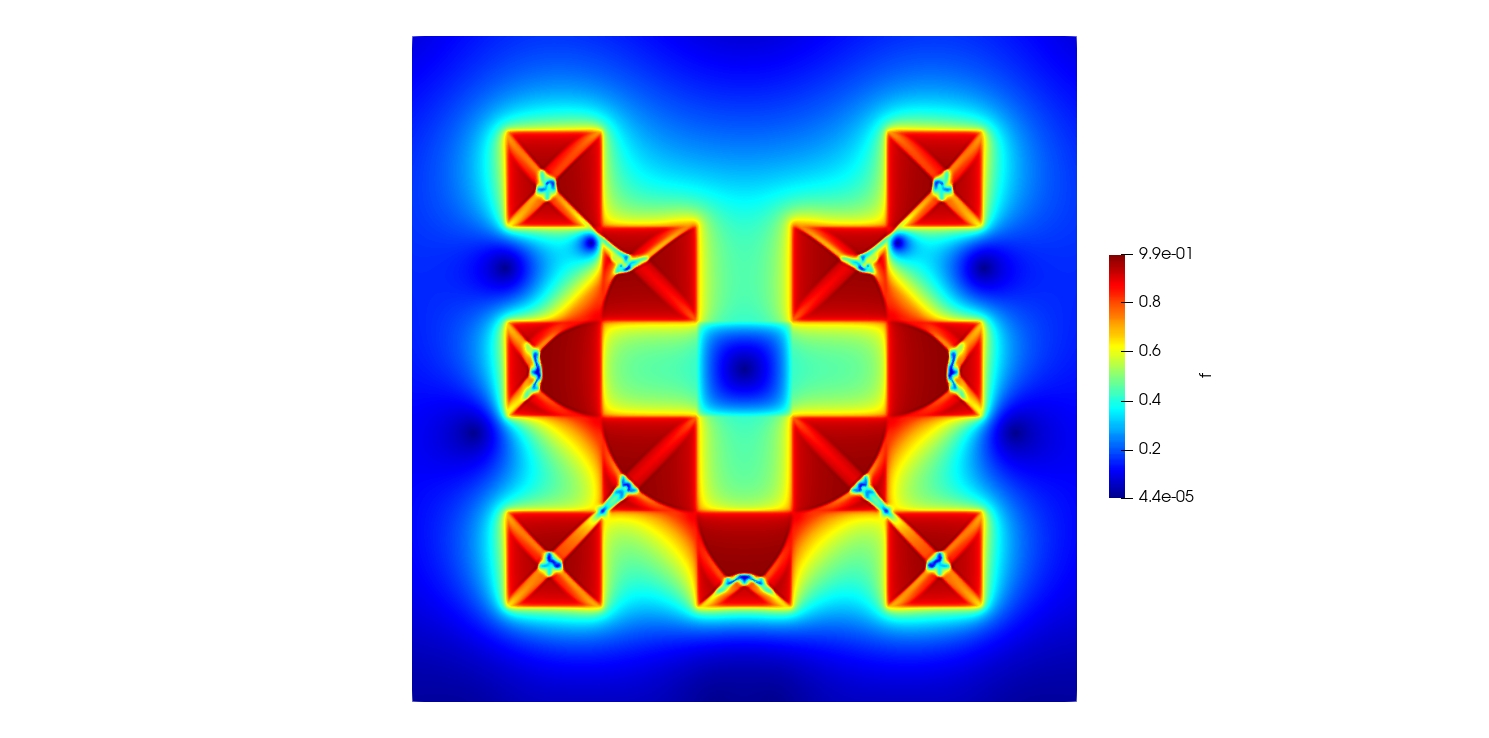}
		\subcaption{$f =\frac{|\po|}{\pz}$, isotropic source~\eqref{eq:lattice_source1}}
	\end{subfigure}
        \begin{subfigure}[c]{0.45\textwidth}
		\includegraphics[trim={14cm 0cm 10.67cm 0cm},clip, width = 0.95\textwidth]{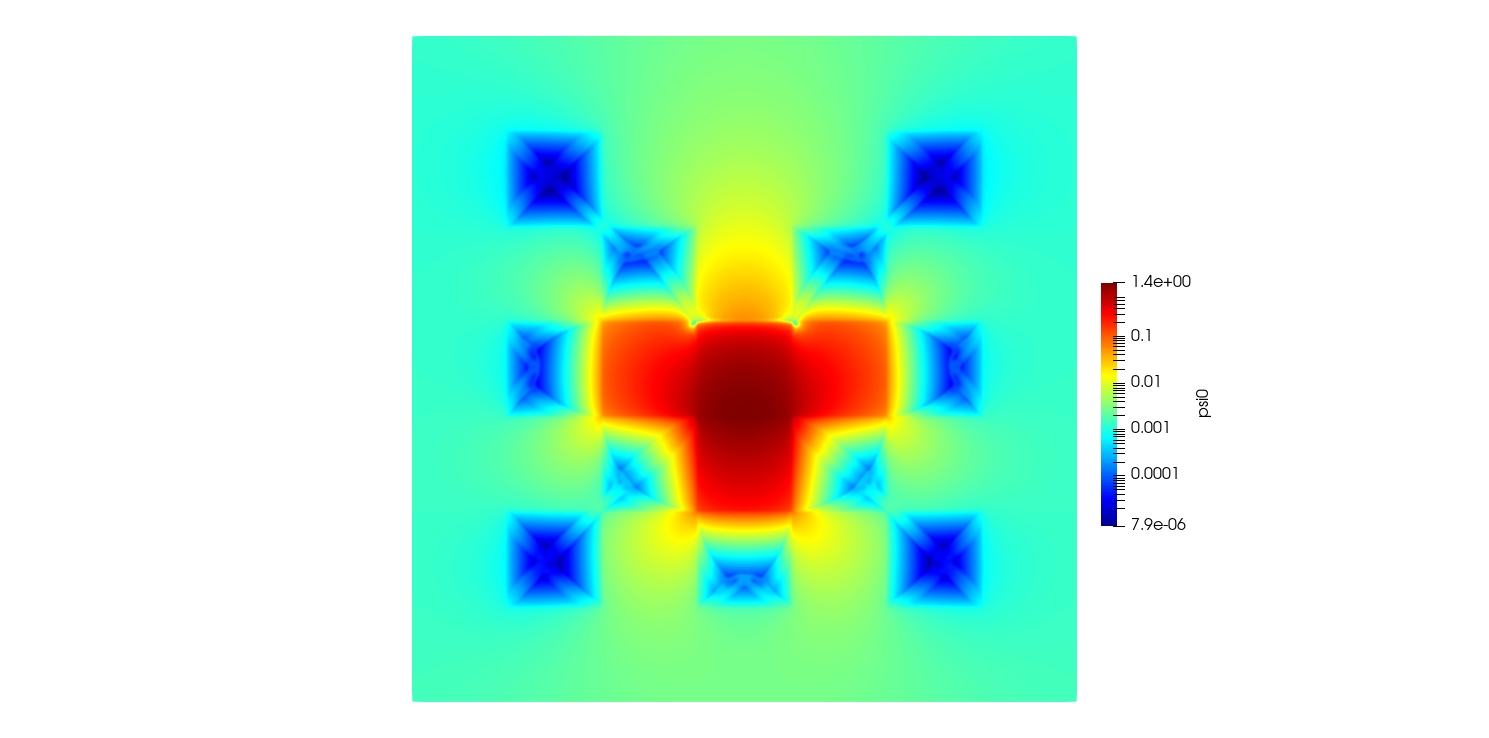}
		\subcaption{$\pz$ (logarithmic scale), anisotropic source~\eqref{eq:lattice_source2}}
	\end{subfigure}
	\begin{subfigure}[c]{0.45\textwidth}
		\includegraphics[trim={14cm 0cm 10.67cm 0cm},clip, width = 0.95\textwidth]{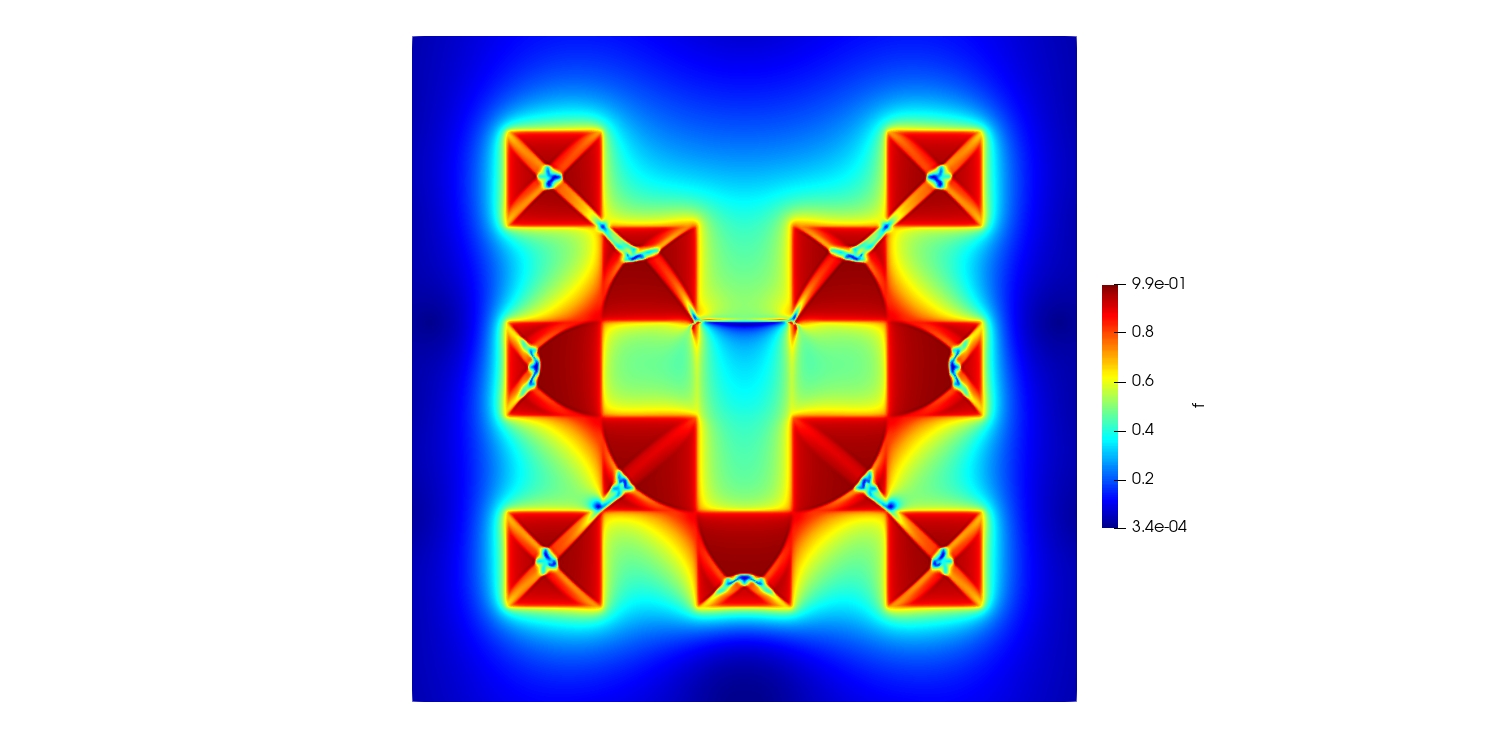}
		\subcaption{$f =\frac{|\po|}{\pz}$, anisotropic source~\eqref{eq:lattice_source2}}
	\end{subfigure}
	\caption{Steady-state Lattice problem computed with the MCL scheme on a uniform rectangular mesh with $N_h = 512^2$ nodes per component and $\mathrm{CFL} = 0.9$. Results for source~\eqref{eq:lattice_source1} (top) and source~\eqref{eq:lattice_source2} (bottom).}
    \label{fig:steadyLattice}
\end{figure}

\begin{figure}[h!]
    \centering
    \includegraphics[trim={0.9cm 0.3cm 2.15cm 1cmm},clip, width = 0.5\textwidth]{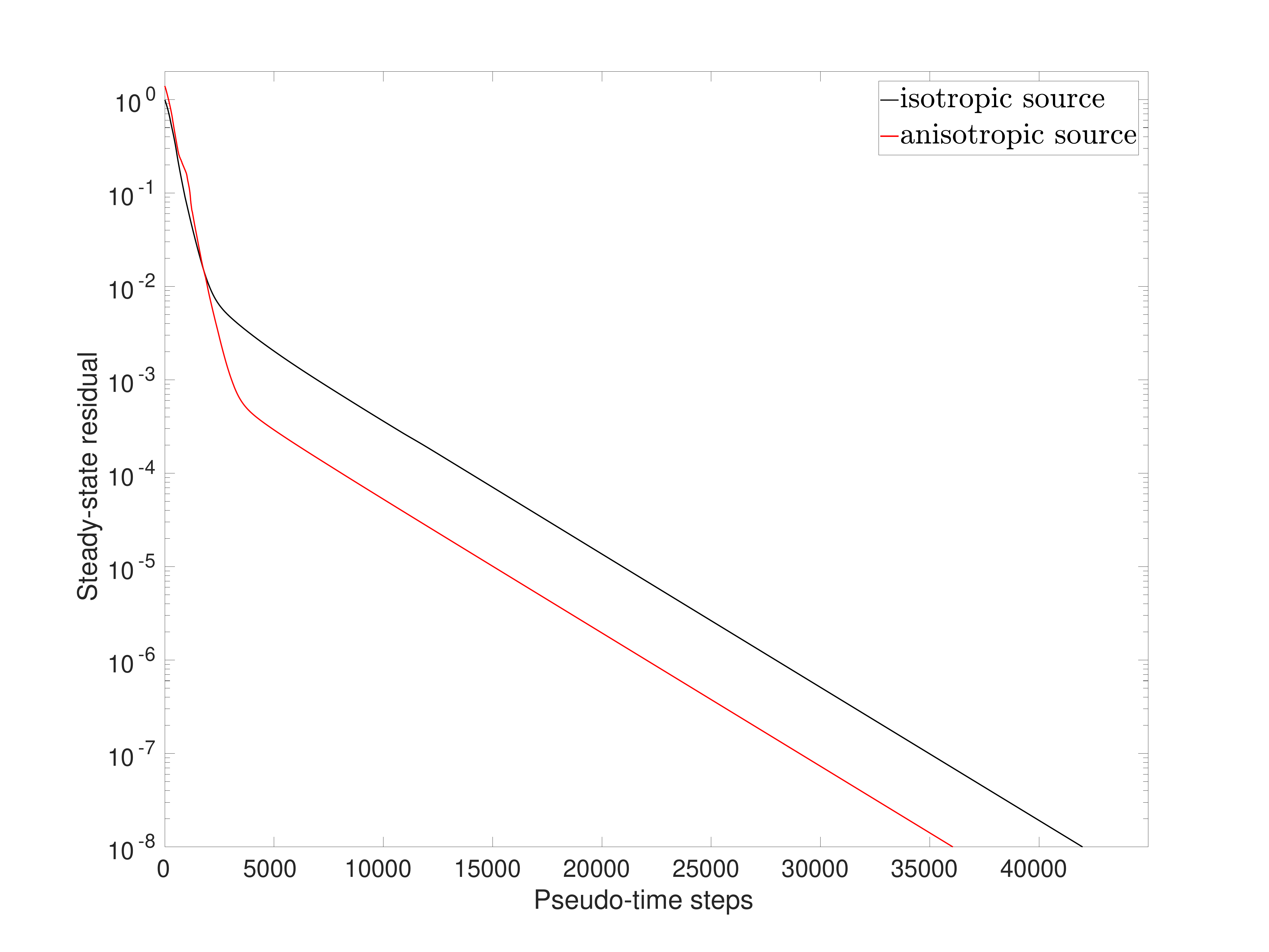}
    \caption{Convergence of steady-state residuals $r_h$ for the MCL discretization of the lattice problem using a uniform rectangular mesh with $N_h = 512^2$ nodes per component and pseudo-time stepping with $\mathrm{CFL} = 0.9$.}
    \label{fig:residual}
\end{figure}

Both setups of this benchmark are very challenging due to the complex structure of the forcing terms.
The strong absorption drives the solution close to the boundary of the realizable set. The particle densities drop below $\pz \leq 10^{-13}$ before the wavefront reaches the absorbing regions.
However, no nonphysical states where detected during any simulation. 
As seen in Figs~\ref{fig:transientLattice} and~\ref{fig:steadyLattice}, the interfaces between the scattering and absorbing media are captured well and no spurious oscillations occur. 

We plot the evolution of the steady-state residuals in Fig.~\ref{fig:residual}.
The initial residual drops rapidly until the radiative wave reaches the boundary of the domain $\mathcal D$.  After that, the MCL solution converges to the steady state in a monotone manner. This behavior demonstrates that the proposed approach is suitable for steady-state computations, the efficiency of which can be enhanced by switching to fully implicit pseudo-time stepping of backward Euler type. The IDP property of implicit MCL schemes can be verified following the analysis performed in \cite{moujaes2025} for the compressible Euler equations.

\section{Discussion}
\label{sec:disc}

Regarding prior studies concerning radiation transport applications, the lack of realizability is an alarming drawback of currently employed deterministic simulation tools, such as standard discrete ordinate / discontinuous Galerkin methods for the LBE \cite{bedford2019,DeMartino2021dose,gifford2006,vassiliev2010}. This has probably hindered the use of such tools in proton therapy. In this field, Monte-Carlo simulations are currently the most accurate method for clinical dose calculations~\cite{JANSON2024,Lin2021,Saini2018,verbeek2021}. Dose engines based on (accurate moment approximations to) the LBE would be ideally suited for secondary in-silico dose checks in the frame of patient-specific quality assurance~\cite{Aitkenhead2020,Magro2022,Dreindl2024}, because their dose calculation algorithm is fundamentally different. Furthermore, LBE-based radiation transport calculations could substantially speed up calculations of the out-of-field dose~\cite{validation}. Eventually, treatment plan optimization could benefit from LBE-based modeling \citep{barnard2012,boman2007radiotherapy}.


\section{Conclusions}
\label{sec:concl}

In this paper, we developed a fail-safe limiting framework for enforcing
realizability in continuous finite element methods for the $M_1$ model
of radiative transfer.
To guarantee preservation of invariant domains by our method,
we analyzed exact solutions of projected Riemann problems
and proved that intermediate states of the homogeneous problem stay in the
convex realizable set. Extending this analysis to the inhomogeneous case,
we found that the fully discrete scheme is provably IDP if the source
terms are included in an implicit manner. To achieve high-order accuracy, we
perform conservative IDP corrections of the low-order intermediate states.
The proposed methodology extends the framework of monolithic convex limiting
to the $M_1$ model of radiative transfer. The results of Lemma~\ref{lem:upm}
and Theorem~\ref{thm:IDPbarstates} carry over naturally to higher-order moment
models ($M_N$ with $N \geq 2$).  This observation opens the possibility of
applying MCL to finite element discretizations of such systems. However,
further efforts need to be invested in the design of realizable closures
and tailor-made limiting techniques for high-order moments, such as the
second-order tensor $\pt$. As discussed in \cite[Section~5.2]{kuzmin2020},
the MCL framework makes it possible
 to constrain the local eigenvalue range in this context.
Additional representatives of limiting approaches for tensor fields can be found in
\cite{lohmann2019}. A further
promising research avenue is experimental validation of moment models
and their practical use for reliable
dose computations in clinical radiotherapy (see the discussion in Section \ref{sec:disc}). 




\section*{Acknowledgments}

The authors are grateful to Dr. Jörg Wulff (West German Proton Therapy Centre Essen) and Prof. Andreas Rupp (Saarland University, Saarbrücken) for helpful remarks regarding physical and computational aspects of radiative transfer modeling.

\bibliographystyle{plain}
\bibliography{bibliography}

\end{document}